\documentclass{article}

\usepackage[english]{babel}
\usepackage[utf8]{inputenc}
\usepackage{amsmath}
\usepackage{amssymb}
\usepackage[normalem]{ulem}
\usepackage{amsthm}
\usepackage{comment}
\usepackage{faktor}
\usepackage{cancel}
\usepackage[toc,page]{appendix}
\usepackage{hyperref}
\hypersetup{%
	pdfpagemode={UseOutlines},
	bookmarksopen,
	pdfstartview={FitH},
	colorlinks,
	linkcolor={black},
	citecolor={black},
	urlcolor={black}
}
\usepackage{amsfonts}
\usepackage{mathrsfs}
\usepackage{physics}
\usepackage{graphicx}
\usepackage{xcolor}
\usepackage[colorinlistoftodos]{todonotes}
\numberwithin{equation}{section}
\newtheorem{theorem}{Theorem}[section]
\newtheorem{corollary}[theorem]{Corollary}
\newtheorem{lemma}[theorem]{Lemma}
\newtheorem{prop}[theorem]{Proposition}
\newtheorem{deff}[theorem]{Definition}
\newtheorem{remark}[theorem]{Remark}
\theoremstyle{definition}

\newtheorem{example}[theorem]{Example}

\newcommand{\lpl}{\left\langle}
\newcommand{\rer}{\right\rangle}

\renewcommand{\theta}{\vartheta}

\newcommand{\eps}{\varepsilon}
\newcommand{\mc}{\mathcal}

\newcommand{\mbf}{\mathbf}

\renewcommand{\l}{\left(}
\renewcommand{\r}{\right)}

\renewcommand{\ll}{\left[}
\newcommand{\rr}{\right]}
\renewcommand{\lll}{\left\{}
\newcommand{\rrr}{\right\}}

\usepackage{arxiv}

\usepackage[utf8]{inputenc} 
\usepackage[T1]{fontenc}    
\usepackage{hyperref}       
\usepackage{url}            
\usepackage{booktabs}       
\usepackage{amsfonts}       
\usepackage{nicefrac}       
\usepackage{microtype}      
\usepackage{lipsum}		
\usepackage{graphicx}
\usepackage{doi}

\title{Instability of boundary layers with the Navier boundary condition}
\date{}

\author{{Lorenzo Quarisa}\thanks{LQ is supported by an EU Chancellor's scholarship from the University of Warwick. } \\
	Mathematics Institute\\
	University of Warwick\\
	Coventry CV47AL, United Kingdom \\
	\texttt{lorenzo.quarisa@warwick.ac.uk}  \\
	\And
{José L. Rodrigo} \\
	Mathematics Institute\\
	University of Warwick\\
	Coventry CV47AL, United Kingdom \\
	\texttt{j.rodrigo@warwick.ac.uk} \\
}

\renewcommand{\shorttitle}{Instability of boundary layers with the Navier boundary condition}

\hypersetup{
pdftitle={Instability of boundary layers with the Navier boundary condition},
pdfauthor={David S.~Hippocampus, Elias D.~Striatum},
pdfkeywords={Boundary layers, Prandtl expansion, zero viscosity limit, Navier friction boundary condition},
}

\begin{document}
    
\maketitle

\begin{abstract}
	 We study the $L^{\infty}$ stability of the Navier-Stokes equations with a viscosity-dependent Navier friction boundary condition  around shear profiles which are linearly unstable for the Euler equation. The dependence from the viscosity is given in the Navier boundary condition as $\partial_y u = \nu^{-\gamma}u$ for some $\gamma\in\mathbb{R}$, where $u$ is the tangential velocity. With the no-slip boundary condition, which corresponds to the limit $\gamma \to +\infty$, a celebrated result from E. Grenier \cite{grenier} provides an instability of order $\nu^{1/4}$. Paddick \cite{paddick} proved the same result in the case $\gamma=1/2$, furthermore improving the instability to order one. In this paper, we extend these two results to all $\gamma \in \mathbb{R}$, obtaining an instability of order $\nu^{\theta}$, where in particular $\theta=0$ for $\gamma \leq 1/2$ and $\theta=1/4$ for $\gamma \geq 3/4$. When $\gamma \geq 1/2$, the result denies the validity of the Prandtl boundary layer expansion around the chosen shear profile.
\end{abstract}

\keywords{Boundary layers \and Prandtl expansion \and Inviscid limit \and Navier friction boundary condition \and Nonlinear instability.  }

\section{Introduction}
A central problem in mathematical fluid dynamics is the approximation of inviscid flows with low viscosity flows, especially when boundaries are present  \cite{E}, \cite{maekawa}. When a boundary is present, the theory is especially suspectible to the type of boundary condition prescribed for the viscous flow. In this paper, we will focus on a \emph{Navier boundary condition}, which allows slip at the boundary with a slip length depending on a power of the viscosity. 

Assume that our spatial domain is two-dimensional with a flat boundary, e.g.  $\mathbb{R}\times \mathbb{R}_+$ or $\mathbb{T}\times \mathbb{R}_+$, and let $\nu>0$ be the viscosity. A viscous flow $\mbf{u}^{\nu}=(u^{\nu}(t,x,y),v^{\nu}(t,x,y))$ is assumed to satisfy the \emph{Navier-Stokes equations} with the \emph{no-slip boundary condition}
\begin{equation}\label{eq:nsnoslip}\begin{cases}\partial_t\mbf{u}^{\nu}+\mbf{u}^{\nu}\cdot \nabla \mbf{u}^{\nu}+\nabla p^{\nu}=\nu \Delta \mbf{u}^{\nu};\\
\nabla \cdot \mbf{u}^{\nu}=0;\\
\mbf{u}^{\nu}=0& \text{at }y=0;\end{cases} \end{equation} 
whereas an inviscid flow $\mbf{u}^E=(u^E(t,x,y),v^E(t,x,y))$ satisfies the \emph{Euler equations}
\begin{equation}\label{eq:euler}\begin{cases}\partial_t\mbf{u}^{E}+\mbf{u}^{E}\cdot \nabla \mbf{u}^{E}+\nabla p^{E}=0;\\
\nabla \cdot \mbf{u}^{E}=0;\\
v^{E}=0& \text{at }y=0.\end{cases}
\end{equation}
The main obstacle to the convergence of $\mbf{u}^{\nu}$ to $\mbf{u}^E$ as $\nu\to 0$ is the contrast between the boundary conditions \eqref{eq:nsnoslip}$_3$ and \eqref{eq:euler}$_3$. 

To overcome this issue, in 1904 Prandtl  \cite{prandtl} proposed the existence of a boundary layer with a  \emph{size} of order $\sqrt{\nu}$ where the fluid undergoes a transition from non-viscous flow to match the boundary condition \eqref{eq:nsnoslip}$_3$. 
In the boundary layer, some new equations for the flow may be derived, called the
\emph{Prandtl equations}.  These are formally obtained by applying the change of variables $y\mapsto \tilde{y}:=y/\sqrt{\nu}$, $(u^{\nu},v^{\nu})\mapsto (u^P,v^P):=\l u^{\nu},v^{\nu}/{\sqrt{\nu}}\r $ to \eqref{eq:nsnoslip}:
\begin{equation}\label{eq:prandtl}\begin{cases}\partial_t u^P + u^P\partial_x u^P+ v^P\partial_{\tilde{y}} u^P+\partial_x p^P = \partial_{\tilde{y}\tilde{y}}u^P;\\
\partial_x u^P+\partial_{\tilde{y}}v^P=0;\\
\lim_{\tilde{y}\to +\infty}u^P=\left.u^E\right|_{y=0};\\
 u^P=v^P=0& \text{ at }\tilde{y}=0.\end{cases}
\end{equation}
Moreover, Prandtl's model predicts that, given
$$\mbf{u}^b=(u^b,\sqrt{\nu}v^b):=(u^P-\left.u^E\right|_{y=0},\sqrt{\nu}v^P) $$
then
the following \emph{boundary layer expansion} holds:
\begin{equation}\label{eq:blexp0}\mbf{u}^{\nu}(t,x,y)\sim \mbf{u}^E(t,x,y)+ \mbf{u}^b\l t,x,\frac{y}{\sqrt{\nu}}\r\qquad \text{ as }\nu \to 0. \end{equation}
A long-standing problem is whether the above formula is mathematically valid. There are different questions which may be formulated in this context. For instance, the problem of the well-posedness of the Prandtl equations \eqref{eq:prandtl} is still open depending on the functional setting - see for instance the review \cite{maekawa}, Section 3.4. In this paper, we will focus on the instability in time of \eqref{eq:blexp0}.

Grenier \cite{grenier} showed the nonlinear instability of the above expansion around a certain class of \emph{shear profiles}. Namely, fix a profile $U_s\in C^{\infty}(\mathbb{R}_+)$, such that $\lim_{y\to +\infty}U_s(y)=U_{\infty}\in \mathbb{R}$. Suppose that $U_s$ is linearly unstable for the Euler equations, meaning there exists an exponentially growing solution to the linearized Euler equations around $U_s$ (see Definition \ref{defunst}). Now let $u_s(t,\tilde{y})$ be the unique evolution of $U_s$ given by the heat equation
$$ \begin{cases} \partial_t u_s(t,\tilde{y})=\partial_{\tilde{y}\tilde{y}}u_s(t,\tilde{y});\\ 
u_s(t,0)=0;\\ 
u_s(0,\tilde{y})=U_s(\tilde{y}).
\end{cases}
    $$
    Notice that substituting $\tilde{y}=y/\sqrt{\nu}$, the above is equivalent to
    $$\begin{cases}  \partial_t u_s\l t,y/\sqrt{\nu}\r=\nu\partial_{yy}u_s\l t, y/\sqrt{\nu} \r;\\ 
u_s(t,0)=0;\\ 
u_s(0,t,y/\sqrt{\nu})=U_s\l t,y/\sqrt{\nu}\r.\end{cases}, $$
which is just the Navier-Stokes equations \eqref{eq:nsnoslip} in the special case of a shear flow solution $\mbf{u}^{\nu}(t,x,y)=(u_s(t,y),0)$.

    We can now state the main result of \cite{grenier}:

\begin{theorem}\label{greniers} Given a smooth shear profile $U_s$, linearly unstable for the Euler equation, and an arbitrary integer $N>0$, there exists a family of solutions $\lll u^{\nu}(t,x,y)\rrr$ to the Navier-Stokes equations such that $$\|u_s(0,y/\sqrt{\nu})-u^{\nu}(0,x,y)\|_{L^{\infty}}\lesssim \nu^N,$$ but \begin{equation}\label{eq:inst}\|u_s(T^{\nu},y/\sqrt{\nu})-u^{\nu}(T^{\nu},x,y)\|_{L^{\infty}}\gtrsim \nu^{1/4}\end{equation} after a time $T^{\nu}\sim \sqrt{\nu}\log \nu \searrow 0$ as $\nu\to 0$.
\end{theorem}
This shows that the expansion \eqref{eq:blexp0} diverges up to order $\nu^{1/4}$ after some small time $T^{\nu}\searrow 0$, regardless of how close the initial conditions are. This tells us that the boundary layer expansion \eqref{eq:blexp0} is unstable, as the Navier-Stokes solutions $u_s$ and $u^{\nu}$ have the same boundary layer expansion at time $t=0$, but they differ by $\nu^{1/4}$ at time $T^{\nu}$.

The above result assumes that the viscous flow satisfies the no-slip boundary condition \eqref{eq:nsnoslip}$_3$. It is therefore natural to ask if the same instability holds when we replace the no-slip condition \eqref{eq:nsnoslip}$_3$ with boundary conditions which produce weaker boundary layers, such as the Navier boundary condition, which was originally proposed by Navier in 1823 \cite{navier}. This boundary condition, like the no-slip condition, forbids penetration of the fluid through the boundary, but allows some slip to occur. In this paper, we will focus on the following viscosity-dependent condition:
\begin{equation}\label{eq:navierbc0}\begin{cases}\partial_y u(t,x,0) = \nu^{-\gamma} u(t,x,0);\\ 
v(t,x,0)=0.
\end{cases}\end{equation}
When $\gamma=0$, so that $\nu^{\gamma}$ - known as the \emph{slip length} - is independent of the viscosity, Iftimie and Sueur \cite{iftimie} proved the validity of the following boundary layer expansion with an \emph{amplitude} of order $\sqrt{\nu}$, which is the factor multiplying $\mbf{u}^b$:
$$\mbf{u}^{\nu}(t,x,y)=\mbf{u}^E(t,x,y)+ \sqrt{\nu}\mbf{u}^b\l t,x,\frac{y}{\sqrt{\nu}}\r +O(\nu).$$
However if $\gamma > 0$, the boundary layers become significant again. Notice that in the limit $\gamma \to +\infty$ we recover the no-slip case. The corresponding appropriate boundary layer expansion, as investigated in \cite{wangwangxin}, is the following for $\gamma \geq 0$:
\begin{equation}\label{eq:expgen}\mbf{u}^{\nu}(t,x,y) \sim \mbf{u}^E(t,x,y)+ \nu^{\max\lll 1/2-\gamma;0\rrr} \mbf{u}^b \l t,x,\frac{y}{\sqrt{\nu}}\r, \end{equation}

which corresponds to the expansion found in \cite{iftimie} when $\gamma=0$. The factor $\displaystyle  \nu^{\max \lll 1/2-\gamma;0\rrr}$ represents the amplitude of the boundary layer. Hence it makes sense to identify $\gamma=1/2$ as the \emph{critical exponent}, where we transition from order one amplitude ($\gamma \geq 1/2$) to a small amplitude which vanishes with the viscosity ($\gamma <1/2$).

In \cite{paddick}, Paddick considers the case $\gamma=1/2$ and proves an instability result akin to \autoref{greniers}. In fact, his result is stronger, as he is able to obtain an order one instability. This is in spite of the fact that the boundary layer is weaker, so one would expect better stability.

The main advantage of the exponent $\gamma=1/2$ is that the evolution of the shear profile $U_s$ is still independent of the viscosity. Namely, we can choose $u_s(t,\tilde{y})$ to solve
$$\begin{cases}\partial_t u_s(t,\tilde{y})= \partial_{\tilde{y}\tilde{y}}u_s(t,\tilde{y});\\
\partial_Y u_s(t,0)=u_s(t,0);\\
u_s(0,\tilde{y})=U_s(\tilde{y});
\end{cases}
$$
which is equivalent, with $\tilde{y}=y/\sqrt{\nu}$, to
$$\begin{cases}\partial_t u_s\l t,y/\sqrt{\nu}\r=\nu \partial_{yy} u_s\l t,y/\sqrt{\nu}\r;\\ 
\partial_y u_s(t,0)=\nu^{-1/2}u_s(t,0); \\ 
u_s\l 0, t,y/\sqrt{\nu}\r=U_s\l t,y/\sqrt{\nu}\r;
\end{cases}
$$
which are precisely the Navier-Stokes equation with the boundary condition \eqref{eq:navierbc0}, $\alpha=\nu^{-1/2}$. If $\gamma \neq 1/2$, then with the above change of variables it is impossible to obtain a unique shear flow $u_s$, independently of the viscosity. In this case, to obtain the correct boundary conditions after the re-scaling, we must define a family of flows $u_s^{\nu}$ satisfying
\begin{equation}\label{eq:shflows}
\begin{cases}\partial_t u_s^{\nu}(t,\tilde{y})=\partial_{\tilde{y}\tilde{y}}u_s^{\nu}(t,\tilde{y});\\
\partial_{\tilde{y}} u_s^{\nu}(t,0)=\nu^{1/2-\gamma} u_s^{\nu}(t,0);\\
u_s^{\nu}(0,\tilde{y})= U_s(\tilde{y}).
\end{cases}
\end{equation}
Because of the $\nu^{1/2-\gamma}$ factor appearing in the boundary condition \eqref{eq:shflows}$_2$, it would be impossible for a single flow, independent from $\nu$, to satisfy \eqref{eq:shflows} for all $\nu$. 

In the main result of this paper, \autoref{main}, we extend Paddick's and Grenier's instability results to all exponents $\gamma \in \mathbb{R}$, while the invalidity of the boundary layer expansion \eqref{eq:expgen} can be generalized to all $\gamma \geq 1/2$. Indeed, the initial condition $U_s(\tilde{y})$ for $u_s^{\nu}$ is a function of $y/\sqrt{\nu}$, so the shear flow $u_s^{\nu}$ can only satisfy the boundary layer expansion \eqref{eq:expgen} at time $t=0$ if the amplitude of the boundary layer is of order one, i.e. if $\gamma \geq 1/2$. In general, we obtain an $L^{\infty}$ instability result for the Navier-Stokes equations which is of order $\nu^{\theta}$, where $\theta$ is a continuous and increasing function of $\gamma$ given by
$$\theta:=\begin{cases} \frac{1}{4} & \gamma \geq \frac{3}{4};\\ 
\gamma - \frac{1}{2} & \frac{1}{2}<\gamma < \frac{3}{4};\\
0 & \gamma \leq \frac{1}{2}.
\end{cases}$$
This interpolates Grenier's original result (limit for $\gamma \to +\infty$) and Paddick's for $\gamma=1/2$. From $\gamma=1/2$ the order of the instability decays until $\gamma=3/4$, where it stabilizes at $\nu^{1/4}$, as in Grenier's case. The reason why this occurs is simply a consequence of the boundary layer expansion \eqref{eq:expgen}, and the use in the proof of an isotropic change of variables mapping $\gamma$ to $2\gamma-3/2$.  

For $\gamma \neq 1/2$, as discussed above, we have to replace the single shear flow $u_s$ in the statement of \autoref{greniers} with a family of viscosity-dependent shear flows $u_s^{\nu}\l t, y/\sqrt{\nu}\r$, defined as the solution of \eqref{eq:shflows}. We refer to Section \ref{main} for the details and the precise statement. 

The main consequence of this choice is that Grenier's method must be complemented with some uniform-in-$\nu$ estimates on $u_s^{\nu}$ (Lemma \ref{usbound} and Lemma \ref{otherbound}). Because we need these uniform bounds to hold in Sobolev spaces of arbitrarily high orders, the shear flows $u_s^{\nu}$ must satisfy the compatibility conditions of all orders at $(t,\tilde{y})=(0,0)$. This is only possible if all the derivatives of $U_s$ vanish at $\tilde{y}=0$, which will have to be added as an assumption. This is a heavy limitation, as it eliminates the possibility of using analytic shear flows. However, this assumption still includes shear flows that are smooth, or in an arbitrary non-analytic Gevrey class. In Section \ref{unstableflows}, we apply a result from \cite{zlin} to show that there exist flows satisfying this assumption which are linearly unstable for the Euler equation. 

We remark that there are also positive results for the validity of the formula \eqref{eq:expgen}. In \cite{iftimie}, the validity is proven when $\gamma=0$; in the more recent paper \cite{tao}, the authors establish it for $\gamma\in (0,1/2]$ and initial data in the Gevrey class $(2\gamma)^{-1}$, which for $\gamma=1/2$ is just the analytic class. However, our result proves the invalidity of the boundary layer expansions only when $\gamma \geq 1/2$, and is therefore not in contradiction with these results.

We also point out that more recently E. Grenier and T. Nguyen proved in \cite{greniernguyen} a stronger result of order one instability in $L^{\infty}$ for the no-slip condition. In light of this, it is likely that the order one instability can be extended for the full range $\gamma \geq 1/2$, using similar techniques. This problem is currently under scrutiny by the authors of this paper.

\textbf{\uline{Organization of the paper.}} In Section \ref{heat}, we gather some preliminary results on the heat equation with a mixed (or 'Robin') boundary condition, which are required in various moments of the proof of our main result. In Section \ref{unstableflows}, we prove that shear profiles satisfying all the assumptions of our main result do exist. Finally, in Section \ref{main0}, we state and prove our main result.

\section{Heat equation with a mixed boundary condition}\label{heat}

Define the \emph{heat kernel} on the half-line
\begin{equation}\label{eq:heatker}K(t,y):=\frac{1}{\sqrt{4\pi t}}e^{-\frac{y^2}{4t}},\qquad t>0,y\geq 0.\end{equation}
Its first $y$-derivative is
\begin{equation}\label{eq:firstkerder}\partial_y K(t,y)=\frac{-2y}{\sqrt{4^3\pi t^3}}e^{-\frac{y^2}{4t}}. \end{equation}

Throughout this section, let $u_0\in C_b^{\infty}(\mathbb{R}_+)$, meaning that $u_0$ is smooth and all its derivatives are bounded in $\mathbb{R}_+$, and let $a\geq 0$.
\begin{prop}\label{heatrobin} Let $u\in C_b^{\infty}([0,T]\times \mathbb{R}_+)$ be the unique smooth and bounded solution of
\begin{equation}\label{eq:heatrobin2}
\begin{cases}\partial_t u = \partial_{yy}u & (t,y)\in \mathbb{R}_+\times \mathbb{R}_+;\\
\partial_y u=a u & y=0;\\
u=u_0 & t=0.
\end{cases}
\end{equation}
Then, if $K$ is the heat kernel defined in \eqref{eq:heatker},  we have
\begin{equation}
    u= K \star \tilde{u}_0, 
\end{equation}
where the convolution is on $\mathbb{R}$ and $\tilde{u}_0$ is the continuous extension of $u_0$ to $\mathbb{R}$ such that  $\partial_y \tilde{u}_0-a\tilde{u}_0$ is an odd function; in particular for all $y\geq 0$,
\begin{equation}\label{eq:extension}    \tilde{u}_0(-y)= e^{-ay}\l u_0(0) + \int_0^y e^{a\tilde{y}}\l u'_0(\tilde{y})-au_0(\tilde{y}) \r\,\mathrm{d}\tilde{y}\r.
\end{equation}
\end{prop}
\begin{proof}First we check that $\tilde{u}_0$ defined by \eqref{eq:extension} satisfies the required property. We have
\begin{align*}
    \partial_y (\tilde{u}_0(-y))= -a \tilde{u}_0(-y)+ u'_0(y)-au_0(y), 
\end{align*}
so
$$ u'_0(y)-au_0(y) = -\l \tilde{u}'_0(-y)-a\tilde{u}_0(-y)\r.$$
Thus $\partial_y \tilde{u}_0-a\tilde{u}_0$ is odd.

Equations \eqref{eq:heatrobin2}$_1$ and \eqref{eq:heatrobin2}$_3$ follow from the usual properties of the heat kernel and convolution. It remains to check \eqref{eq:heatrobin2}$_2$. At $y=0$, we have
\begin{align*}
    (\partial_y u -au)(t,0)=\int_{\mathbb{R}}K(t,x)\l  \tilde{u}'_0(-x)-a\tilde{u}_0(-x)\r\,\mathrm{d}x.
\end{align*}
But we know that $\partial_y \tilde{u}_0-a\tilde{u}_0$ is odd, whereas $x\mapsto K(t,x)$ is even, therefore the product is odd and the integral equates zero. Hence \eqref{eq:heatrobin2}$_2$ is satisfied for all $t\geq 0$. 
\end{proof}
\begin{remark}
By integration by parts, \eqref{eq:extension} can be rewritten as
\begin{equation}\label{eq:extension2}
\tilde{u}_0(-y)= u_0(y) -2a\int_0^y e^{-a(y-\tilde{y})}u_0(\tilde{y})\,\mathrm{d}\tilde{y},
\end{equation}
or
\begin{equation}\label{eq:extension3}
    \tilde{u}_0(-y)= -u_0(y) +2e^{-ay}u_0(0)+2\int_0^y e^{-a(y-\tilde{y})}u'_0(\tilde{y})\,\mathrm{d}\tilde{y}.
\end{equation}
\end{remark}

\begin{remark}\label{limits}
If we take the limit for $a\to +\infty$ in \eqref{eq:extension3} we obtain 
$$\tilde{u}_0(y)=-u_0(-y), $$
which is just the odd extension of $u_0$. This choice of extension yields the solution to the Dirichlet problem, which is in agreement with the fact that the condition $\partial_y u= au$ at $y=0$ converges to the Dirichlet condition if $a\to +\infty$. We call this solution $u^{\infty}_0$. 

Taking $a=0$ instead, from \eqref{eq:extension2} we see that $\tilde{u}_0$ is the even extension, which corresponds to the Neumann problem $u'_0(y=0)=0$.
\end{remark}

From now on, \uline{we will assume that $u_0(0)=0$}. Notice that this implies that $u'_0(0)=0$.

In this case we may rewrite \eqref{eq:extension3} as 
$$  \tilde{u}_0(-y)= -u_0(y)+2\int_0^y e^{-a(y-\tilde{y})}u'_0(\tilde{y})\,\mathrm{d}\tilde{y}. $$

We can then obtain the estimate
\begin{align*}
    |\tilde{u}_0(-y)|&\leq |u_0(y)|+2\frac{1-e^{-ay}}{a} \sup_{\tilde{y}\in [0,y]}\left|u'_0(\tilde{y})\right|.
\end{align*}
Similarly from \eqref{eq:extension2} we get
$$|\tilde{u}_0(-y)|\leq |u_0(y)|+ 2(1-e^{-ay})\sup_{\tilde{y}\in [0,y]}|u_0(\tilde{y})|. $$
Putting the two together, we get
\begin{equation}\label{eq:boundshear}
\|\tilde{u}_0\|_{L^{\infty}}\leq \|u_0\|_{L^{\infty}}+ 2\min \lll \frac{1}{a}\|u_0'\|_{L^{\infty}};\|u_0\|_{L^{\infty}}\rrr.
\end{equation}
Note that both when $a\to 0$ and $a\to +\infty$, the contribution of the derivative $u_0'$ in the estimate disappears.

 We want to find an explicit  expression of the derivatives of $\tilde{u}_0$, starting from \eqref{eq:extension3}.
\begin{prop}[Derivatives formula] We have for all $k\in \mathbb{Z}_{\geq 0}$, and for all $y>0$,
\begin{align}\label{eq:derivatives}
    (-1)^k\tilde{u}_0^{(k)}(-y)&= -u_0^{(k)}(y)+2e^{-ay}\sum_{j=0}^{k}(-a)^{k-j} u_0^{(j)}(0)+2\int_0^y e^{-a(y-\tilde{y})}u_0^{(k+1)}(\tilde{y})\,\mathrm{d}\tilde{y}\\  \label{eq:derivatives3}
    &=u_0^{(k)}(y)+2e^{-ay}\sum_{j=0}^{k-1} (-a)^{k-j}u_0^{(j)}(0)-2a\int_0^y e^{-a(y-\tilde{y})}u_0^{(k)}(\tilde{y})\,\mathrm{d}\tilde{y}.
\end{align}
\end{prop}
\begin{proof}
The case $k=0$ is \eqref{eq:extension3}. Suppose \eqref{eq:derivatives} holds for some $k\geq 0$. Then after differentiating both sides, and integrating by parts, we get
\begin{align*}
    (-1)^{k+1}\tilde{u}_0^{(k+1)}(-y)=-{u}_0^{(k+1)}(y)+2e^{-ay}\sum_{j=0}^{k}(-a)^{k+1-j}u_0^{(j)}(0)+ (\star),
\end{align*}
where
\begin{align*}
    (\star)= 2 \l u_0^{(k+1)}(y)-u_0^{(k+1)}(y)+e^{-ay}u_0^{(k+1)}(0)+\int_0^ye^{-a(y-\tilde{y})}u_0^{(k+2)}(\tilde{y})\,\mathrm{d}\tilde{y}\r.
\end{align*}
From this, \eqref{eq:derivatives} follows, and \eqref{eq:derivatives3} by integration by parts.
\end{proof}
\subsection{Well-posedness of the inhomogeneous boundary problem}\label{firstpart}
In this Section, we prove a well-posedness result for the fully inhomogeneous heat equation with a Dirichlet, Neumann or mixed boundary condition. As long as the data satisfy an exponential-type bound in time, then the solution will satisfy the same bound, without any additional growth in time.  
\begin{lemma}\label{asympt} Let $\alpha>0$ and $\beta \geq 0$. Then there is a constant $C>0$ such that 
$$\int_0^t \frac{e^{\alpha s}}{(1+s)^{\beta}}\,\mathrm{d}s \leq C\frac{e^{\alpha t}}{(1+t)^{\beta}},\qquad \forall t\geq 0.$$
\end{lemma}
\begin{proof} Let $\varphi(t):=\frac{e^{\alpha t}}{(1+t)^{\beta}}$. Since $t\mapsto \frac{\int_0^t\varphi(s)\,\mathrm{d}s}{\varphi(t)} $ is continuous over $[0,\infty)$, we just need to prove that
$$\lim_{t\to  +\infty}\frac{\int_0^t\varphi(s)\,\mathrm{d}s}{\varphi(t)}<\infty. $$
In fact, by L'Hopital's rule, we see that
$$\lim_{t\to  +\infty}\frac{\int_0^t\varphi(s)\,\mathrm{d}s}{\varphi(t)}=\lim_{t\to +\infty}\frac{\varphi(t)}{\varphi(t)\l \alpha-\beta(1+t)^{-1}\r}=\frac{1}{\alpha}. $$
\end{proof}
The following result is taken from \cite{paddick}, Lemma 3.4. We further include a proof.
\begin{lemma}\label{gronwall}Let $\varphi:\mathbb{R}_+\to \mathbb{R}$ satisfy
\begin{equation}\label{eq:gronw}\varphi'(t)\leq \lambda \varphi(t)+ C\frac{e^{\alpha t}}{(1+t)^{\beta}},\qquad \forall t\geq 0, \end{equation}
where $0\leq \lambda <\alpha $, $\beta>0$ and $C>0$ can depend on $\alpha,\beta$ but not on $t$.  Then 
$$\varphi(t)\leq C'\frac{e^{\alpha t}}{(1+t)^{\beta}},\qquad \forall t\geq 0.$$
\end{lemma}
\begin{proof}Integrating \eqref{eq:gronw} from $0$ to $t$, we get
$$\varphi(t) \leq \varphi(0)+\int_0^t \lambda\varphi(s)\,\mathrm{d}s + C\int_0^t \frac{e^{\alpha s}}{(1+s)^{\beta}}\,\mathrm{d}s. $$
Hence, by an application of the Gronwall's lemma (see \cite{piccinini}, page 19) and Lemma \ref{asympt} - since $\alpha-\lambda >0$ - we have
\begin{align*}\varphi(t) &\leq Ce^{\lambda t}\l \varphi(0)+\int_0^t e^{(\alpha-\lambda)s}(1+s)^{-\beta}\,\mathrm{d}s\r\\ 
&\leq Ce^{\lambda t}\cdot C_{\alpha-\lambda,\beta}\frac{e^{(\alpha-\lambda) t}}{(1+t)^{\beta}}\\
& \leq C'\frac{e^{\alpha t}}{(1+t)^{\beta}}.
 \end{align*}
\end{proof}

Hereafter, we will denote with $C_k$ an arbitrary positive constant depending on $k$, which may vary from line to line. These constants are always independent from $t$ and $y$.
\begin{prop}[$H^k$ estimates]\label{heatestimate} Let $a,b\geq 0$, $(a,b)\neq (0,0)$.
Suppose that $u_0,f\in H^k(\mathbb{R}_+)$ and $r\in H^k(\mathbb{R}_+\times \mathbb{R}_+)$ for all $k\in \mathbb{Z}_{\geq 0}$, and there exist $\alpha>0$, $\beta\geq 0$ such that
$$\|\partial_y^kr(t)\|_{L^2(\mathbb{R}_+)}+|f^{(k)}(t)|\leq C_k \frac{e^{\alpha t}}{(1+t)^{\beta}}\qquad \forall t\geq 0,k\in \mathbb{Z}_{\geq 0}.$$
Let $u(t,y)$ be the classical solution of 
\begin{equation}\label{eq:heatdef}
    \begin{cases}
    \partial_t u(t,y)=\partial_{yy}u(t,y)+r(t,y) & t\geq 0,y\geq 0\\ 
    a\partial_y u(t,0)-bu(t,0)=f(t) & t\geq 0\\ 
    u(0,y)=u_0(y) & y\geq 0.
    \end{cases}
\end{equation}
Then, for all $k\in \mathbb{Z}_{\geq 0}$, we have
\begin{equation}\label{eq:esstimate}\|u(t)\|_{H^k(\mathbb{R}_+)}\leq CC_k\frac{e^{\alpha t}}{{(1+t)^{\beta }}}\qquad \forall t\geq 0,\end{equation}
where $C$ is independent from $r$, $f$ or $u_0$.
\end{prop}
\begin{proof}We only consider the case $k=0$, as the others follow by applying the result to $\partial_y^{2k}u$, using the compatibility conditions to derive the appropriate boundary conditions and interpolating for odd derivatives. Differentiating the energy, we get
\begin{align*}\frac{1}{2}\dv{}{t}\|u(t)\|_{L^2}^2&=-\|\partial_y u(t)\|_{L^2}^2-u(t,0)\partial_y u(t,0)+ \lpl r(t),u(t)\rer_{L^2}.
\end{align*}
Now there are two different cases.
\begin{itemize}
    \item $a\neq 0,b\geq 0$. We can replace $\partial_y u(t,0)=\frac{1}{a}\l f(t,0)+bu(t,0) \r$, so that
    \begin{align*} \frac{1}{2}\dv{}{t}\|u(t)\|_{L^2}^2 &\leq-\|\partial_y u(t)\|_{L^2}^2-\frac{b}{a}|u(t,0)|^2-\frac{1}{a}u(t,0)f(t,0)+\lpl r(t),u(t)\rer_{L^2}\\ 
    &\leq -\|\partial_y u(t)\|_{L^2}^2 +\frac{\eps}{a}|u(t,0)|^2+\frac{1}{4\eps a}|f(t,0)|^2+\lpl r(t),u(t)\rer_{L^2}\\
    &\leq -\|\partial_y u(t)\|_{L^2}^2+C\frac{\eps}{a}\l \|u(t)\|_{L^2}^2+\|\partial_y u(t)\|_{L^2}^2\r+ \frac{1}{4\eps a}|f(t,0)|^2+\lpl r(t),u(t)\rer_{L^2}\\ 
    &\leq \eps \l \frac{C}{a}+1\r\|u(t)\|_{L^2}^2+\frac{1}{4\eps a}|f(t,0)|^2+\frac{1}{4\eps}\|r(t)\|_{L^2}^2\\ 
    &\leq \eps\l \frac{C}{a}+1\r\|u(t)\|_{L^2}^2+ C_{\eps}\frac{e^{\alpha t}}{(1+t)^{\beta}}.\end{align*}
    where we have chosen $\eps>0$ small enough so that $\eps\dfrac{C}{a}<1$ where $C$ is the constant of the embedding $H^1(\mathbb{R}_+)\hookrightarrow L^{\infty}(\mathbb{R}_+)$. Further requiring $\eps\l \dfrac{C}{a}+1\r<\alpha$,  we can apply Lemma \ref{gronwall} and conclude.
    \item $a=0$:  If $u(t,y)$ solves the Dirichlet problem, then 
    $$v(t,y)=\int_0^y u(t,x)\,\mathrm{d}x+\int_0^t\partial_yu(s,0)\,\mathrm{d}s, $$
    solves the Neumann problem with boundary function $f(t)$, remainder $\int_0^y r(t,x)\,\mathrm{d}x$ and initial condition $\int_0^y u_0(x)\,\mathrm{d}x$. Therefore, \eqref{eq:esstimate} holds for $v$, and hence for $u=\partial_y v$.
\end{itemize}
\end{proof}

We will now look for pointwise estimates of the solutions, describing the decay at infinity. Recall the expression of the heat kernel \eqref{eq:heatker}. 
We can now exhibit the explicit solution of the inhomogeneous Dirichlet problem for the heat equation. 
\begin{lemma}\label{heatdir}
   The classical solution $u(t,y)$ of the problem
\begin{equation}\label{eq:heatdir}  \begin{cases}
    \partial_t u(t,y)=\partial_{yy}u(t,y)+r(t,y) & t\geq 0,y\geq 0\\ 
    u(t,0)=f(t) & t\geq 0\\ 
    u(0,y)=u_0(y) & y\geq 0.
    \end{cases}\end{equation}
is given by $u=u_1+u_2+u_3$, where
\begin{align*}
    u_1(t,y)&=-2\partial_y K(t,y)\star f(t)=-2\int_0^t \partial_yK(t-s,y)f(s)\,\mathrm{d}s,\\ 
    u_2(t,y)&=K(t,y)\star \tilde{u}_0(y),\\ 
    u_3(t,y)&=K(t,y)\star_{t,y}\tilde{r}(t,y)=\int_0^t\int_{\mathbb{R}} K(t-s,y-x)\tilde{r}(s,x)\,\mathrm{d}x\,\mathrm{d}s,
\end{align*}
and $\tilde{u}_0$ and $\tilde{r}$ are the odd extensions in the $y$ variable of $u_0$ and $r$ respectively.
\end{lemma}
\begin{proof}It can be checked by plugging $u_1+u_2+u_3$ into \eqref{eq:heatdir}. 
\end{proof}
\begin{prop}[Pointwise estimates]\label{pointwise}Suppose that $u_0\in H^k(\mathbb{R}_+)$ for all $k\in \mathbb{Z}_{\geq 0}$, and there exist $\alpha>0$, $\beta\geq 0$, and a constant $\lambda>0$ such that
$$\begin{cases}|f^{(k)}(t)|\leq C_k \frac{e^{\alpha t}}{(1+t)^{\beta}},\\ 
|\partial_y^k r(t,y)|\leq C_k\frac{e^{\alpha t}}{(1+t)^{\beta}}e^{-\lambda y},\\
|\partial_y^k u_0(y)|\leq  C_k e^{-\lambda y},
\end{cases} \qquad \forall t\geq 0,k\in \mathbb{Z}_{\geq 0}.$$
Let $u(t,y)$ be unique smooth and bounded solution of \eqref{eq:heatdef}.
Then there exists a constant $\mu>0$ independent from time such that
$$|\partial_y^k u(t,y)|\leq C C_k\frac{e^{\alpha t}}{(1+t)^{\beta}}e^{-\mu y} \qquad \forall t\geq 0,y\geq 0, k\in \mathbb{Z}_{\geq 0}, $$
where $C$ is independent from $f,r$ or $u_0$.
\end{prop}
\begin{proof} 
By Proposition \ref{heatestimate}, using Sobolev embeddings we know that
$$| u(t,0)|\leq CC_k\frac{e^{\alpha t}}{(1+t)^{\beta}}\qquad \forall t\geq 0. $$
We can now see $u(t,y)$ as the solution of the Dirichlet problem \eqref{eq:heatdir} with a boundary condition $f(t)=u(t,0)$ satisfying the above estimate. In this case, let $\tilde{u}_0$ and $\tilde{r}_0$ be the odd extensions of $u_0$ and $r_0$ respectively. \\

Write $u=u_1+u_2+u_3$ in the notation of Lemma \ref{heatdir}. Recall that
$$u_1(t,y)=\int_0^t -2\partial_y K(s,y)u(t-s,0)\,\mathrm{d}s. $$
We then have 
\begin{align*}\l\frac{e^{\alpha t}}{(1+t)^{\beta}}\r^{-1}|\partial_y^k u_1(t,y)|&\leq \int_0^t 2|\partial_y^{k+1} K(s,y)|\frac{e^{-\alpha s}}{(1-s/(1+t))^{\beta}}\,\mathrm{d}s\\
&\leq C_{\alpha,\beta}\int_0^{\infty}2|\partial_y^{k+1} K(s,y)|e^{-\frac{\alpha}{2}s}\,\mathrm{d}s\\
&=C_{\alpha,\beta}\int_0^{\infty} \frac{p(s,y)}{s^{(3+2k)/2}}e^{-\frac{y^2}{8s}}e^{-\frac{1}{8s}\l y^2+ 4\alpha s^2\r}\,\mathrm{d}s\\ 
&\leq C_{\alpha,\beta}q(y)e^{-\sqrt{\alpha} y/2}\leq Ce^{-\sqrt{\alpha}y/3}.
\end{align*}
where $p$ and $q$ are polynomials, and in the second-to-last step we simply used the inequality $-A^2-B^2\leq -2AB$. Consider now $u_2$ and $u_3$. We have
\begin{align*}
   \partial_y^ku_2(t,y)= \partial_y^kK(t,y)\star \tilde{u}_0(y)\leq \frac{e^{\alpha t}}{(1+t)^{\beta}}\int_{\mathbb{R}}\partial_x^kK(t,x)e^{-\lambda|y-x|}\,\mathrm{d}x\leq C \frac{e^{\alpha t}}{(1+t)^{\beta}}e^{-\lambda y},
\end{align*}
and finally
\begin{align*}
    \partial_y^ku_3(t,y):=\partial_x^kK(t,y)\star_{t,y}\tilde{r}(t,y)&\leq  \int_{\mathbb{R}}\int_0^t\partial_x^k K(t-s,x)\frac{e^{\alpha s}}{(1+s)^{\beta}}\,\mathrm{d}s\,e^{-\lambda|y- x|}\,\mathrm{d}x\\ 
    &\leq C'_{\alpha,\beta}\frac{e^{\alpha t}}{(1+t)^{\beta}}\int_{\mathbb{R}} e^{-\mu x}e^{-\lambda |y-x|}\,\mathrm{d}x\\
    &\leq C'_{\alpha,\beta}e^{-\mu' y}\frac{e^{\alpha t}}{(1+t)^{\beta}}.
\end{align*}
where $\mu'<\min \lll \mu;\lambda\rrr.$

Since all the partial  solutions $u_1,u_2,u_3$ satisfy the required estimate, then the full solutions $u=u_1+u_2+u_3$ also does. 
\end{proof}
\subsection{Asymptotic behavior as $a\to +\infty$}\label{secondpart}

From now on, we will drop the tilde notation and use $u_0^a$, for $a\in (0,\infty]$, to denote the extension of $u_0$ relative to the coefficient $a$ in the mixed condition \eqref{eq:heatrobin2}$_2$. Recall that $u_0^{\infty}$ is just the odd extension. We are interested in the asymptotic behavior as $a\to +\infty$.
\begin{lemma}[$L^{\infty}$ estimate]We have

$$ \|u^a_0-u^{\infty}_0\|_{L^{\infty}}= O(a^{-1}),\qquad a\to +\infty.
$$
\end{lemma}

\begin{proof}By \eqref{eq:extension2} we have, for all $y\geq 0$,
\begin{align*}
    u^a_0(y)-u^{\infty}_0(y)&=0,\\ 
    u^a_0(-y)-u^{\infty}_0(-y)&=2\int_0^y e^{-a(y-\tilde{y})}u_0'(\tilde{y})\,\mathrm{d}\tilde{y}.
\end{align*}
Hence
$$
    \left|u^a_0(y)-u^{\infty}_0(y)\right|\leq 2a^{-1}\sup_{[0,y]}|u_0'|,
$$
from which the estimate follows.
\end{proof}
\begin{lemma}[$L^1$ estimate] \label{l1est} Assume that $u_0' \in L^1(\mathbb{R}_+)$. Then $u^a_0-u^{\infty}_0\in L^1(\mathbb{R}_+)$ and
$$\|u^a_0-u^{\infty}_0\|_{L^1}=O(a^{-1}),\qquad a\to +\infty. $$
\end{lemma}
\begin{proof} We have
\begin{align*}
    \int_{\mathbb{R}_+}\left|u^a_0(-y)-u^{\infty}_0(-y)\right|&\leq 2\int_0^{\infty}\int_0^y e^{-a(y-\tilde{y})}|u_0'(\tilde{y})|\,\mathrm{d}\tilde{y}\,\mathrm{d}y=(\star).
\end{align*}
Because the function is non-negative, by Tonelli's theorem we can swap the order of integration:
\begin{align*}
    (\star)=2\int_0^{\infty}e^{a\tilde{y}} \int_{\tilde{y}}^{\infty}e^{-ay}\,\mathrm{d}y\,|u_0'(\tilde{y})|\,\mathrm{d}\tilde{y}=a^{-1}\int_0^{\infty}|u_0'(\tilde{y})|\,\mathrm{d}\tilde{y}=O(a^{-1}).
\end{align*}

\end{proof}
\begin{corollary}[$L^p$ estimate] Assume that $u_0'\in L^1(\mathbb{R}_+)$. Then for any $p\in [1,\infty]$, we have $u^a_0-u^{\infty}_0\in L^p(\mathbb{R}_+)$ and
$$\|u^a_0-u^{\infty}_0\|_{L^p}=O(a^{-1}),\qquad a\to +\infty. $$
\end{corollary}
\begin{proof}
Follows from the previous two lemmas by interpolation.

Alternatively, we could notice that $\partial^k u_0^a(y)-\partial^k u_0^{\infty}(y)= 2e^{-a\tilde{y}}\star u_0^{(k+1)}(\tilde{y})$, where the convolution is over $[0,y]$. Applying Young's inequality for convolutions, we obtain for all $p\in [1,\infty]$:
$$\|\partial^ku_0^a-\partial^ku_0^{\infty}\|_{L^p}\leq 2 \|e^{-ay}\|_{L^1}\|u_0^{(k+1)}\|_{L^p}=2a^{-1}\|u_0^{(k+1)}\|_{L^p}=O(a^{-1}). $$
\end{proof}

We now look for higher order estimates. From the expression \eqref{eq:derivatives3} we deduce that
$$(-1)^k\lim_{y\to 0^+}\partial^ku_0^a(-y)= \partial^k u_0(0)+2\sum_{j=0}^{k-1}(-a)^{k-j}\partial^j u_0(0). $$
This tells us that $\tilde{u}_0$ is continuous. But for higher derivatives, the gap between $\tilde{u}_0^{(k)}$ and the odd extension of $u_0$ increases with $a$. In general, we have
\begin{align*}
    \partial^k u_0^a(-y)-\partial^ku_0^{\infty}(-y)=2e^{-ay}\sum_{j=0}^k(-a)^{k-j}u_0^{(j)}(0)+2\int_0^{y}e^{-a(y-\tilde{y})}u_0^{(k+1)}(\tilde{y})\mathrm{d}\tilde{y}.
\end{align*}
In particular, due to the discontinuity at $y=0$, the derivative $\partial^k u_0^a$ does not even belong to $L^p(\mathbb{R})$ unless $u_0^{(j)}(0)=0$ for all $j=0,\dots, k-1$. Requiring the weaker condition $u_0^{(2k+1)}(0)=au_0^{(2k)}(0)$ for all $k\in \mathbb{Z}_{\geq 0}$, which corresponds to enforcing the compatibility conditions of all orders, would also get rid of the problematic terms. However, that can only be true for multiple values of $a$ if $u_0^{(2k+1)}(0)=u_0^{(2k)}(0)=0$.

We will therefore work under the following assumption:
\begin{equation}\label{eq:assumpt01}
    \begin{cases}\partial^k u_0(0)=0, \\
    \partial^{k+1} u_0 \in L^1(\mathbb{R}_+),
    \end{cases}\qquad \forall k\geq 0.
\end{equation}
The first condition implies that \eqref{eq:derivatives} can be simplified to 
\begin{equation}\label{eq:derivatives2}
    (-1)^k\tilde{u}_0^{(k)}(-y)= -u_0^{(k)}(y)+2\int_0^y e^{-a(y-\tilde{y})}u_0^{(k+1)}(\tilde{y})\,\mathrm{d}\tilde{y}.
\end{equation}
In particular, $\tilde{u}_0\in C^{\infty}(\mathbb{R})$. This is crucial as it allows us to consider the derivatives of $\tilde{u}_0$ in Sobolev spaces.

The second condition, together with $u_0\in C_b^{\infty}(\mathbb{R}_+)$, implies that $u_0\in W^{\infty,p}$ for any $p\in [1,\infty]$, in particular $u_0\in H^s$ for all $s\geq 0$. Furthermore, since $u_0\in W^{k+1,1}(\mathbb{R}_+)$ implies $\lim_{y\to +\infty}\partial^k u_0(y)=0$, we deduce that 
$$\lim_{y\to +\infty}u_0(y)=c\in \mathbb{R}, \qquad \lim_{y\to +\infty}\partial^k u_0(y)=0\;\forall k\geq 1. $$

Unfortunately, these assumptions also mean that this we have to exclude profiles $u_0$ which are analytic on $\mathbb{R}_+$ (except for $u_0=0$). However, we can include profiles in the Gevrey class $G^{\gamma}$ for all $\gamma>1$, such as \( \displaystyle u_0(y)= e^{-y^{-(\gamma-1)^{-1}}}. \)

\begin{prop}\label{convid}Assuming \eqref{eq:assumpt01}, for all $k\in \mathbb{Z}_{\geq 0}$ and $p\in [1,\infty]$ we have
\begin{equation}\label{eq:convid}\|u^a_0-u^{\infty}_0\|_{W^{k,p}}=O(a^{-1}),\qquad a\to +\infty. \end{equation}
\end{prop}
\begin{proof}By \eqref{eq:derivatives2}, for any $k\in \mathbb{Z}_{\geq 0}$ we have
\begin{align*}
    \|\partial^k u_0^a-\partial^k u_0^{\infty}\|_{L^{\infty}}&\leq 2 \|\partial^{k+1}u_0\|_{L^{\infty}}\sup_{y\geq 0}\int_0^y e^{-a(y-\tilde{y})}\,\mathrm{d}\tilde{y}\\
    &\leq \frac{2}{a}\|\partial^{k+1}u_0\|_{L^{\infty}}=O(a^{-1}),
\end{align*}
and
\begin{align*}
    \|\partial^k u^a_0-\partial^k u^{\infty}_0\|_{L^1}&\leq 2\int_0^{\infty}\int_0^y e^{-a(y-\tilde{y})}|\partial^{k+1}u_0(\tilde{y})|\,\mathrm{d}\tilde{y}\,\mathrm{d}y\\ 
    &= 2\int_0^{\infty}e^{a\tilde{y}}\int_{\tilde{y}}^{\infty} e^{-ay}\,\mathrm{d}y\,|\partial^{k+1}u_0(\tilde{y})|\,\mathrm{d}\tilde{y}\\ 
    &=\frac{2}{a}\|\partial^{k+1}u_0\|_{L^1}=O(a^{-1}).
\end{align*}
By interpolation, \eqref{eq:convid} follows. 
\end{proof}
Now we state the corresponding result for the solution at any time. Let us call $u^a$ and $u^{\infty}$, respectively, the evolutions of the initial datas $u_0^a$ and $u_0^{\infty}$ under \eqref{eq:heatrobin2}, given by convolution with the kernel $K$. 
\begin{corollary}\label{conv} Assuming \eqref{eq:assumpt01}, for all $k\geq 0$ and $p\in [1,\infty]$ we have
\begin{equation}\label{eq:conv}\|u^a-u^{\infty}\|_{W^{k,p}}=O(a^{-1}),\qquad a\to +\infty, \end{equation}
uniformly for all $t\geq 0$.
\end{corollary}
\begin{proof}We have
$$\|u^a-u^{\infty}\|_{W^{k,p}}=\|K\star (u_0^a-u_0^{\infty})\|_{W^{k,p}}\leq \|K\|_{L^1}\|u_0^a-u_0^{\infty}\|_{W^{k,p}}.$$
Since $\|K(t)\|_{L^1}=1$ for all $t\geq 0$, the thesis follows from  Proposition \ref{convid}. 

\end{proof}
\subsection{Asymptotic behavior as $a\to 0^+$}\label{lim0}
In this subsection we will consider the limit $a\to 0^+$. Recall that for $a=0$, the function $u_0^0$ is the even extension of $u_0$ from $\mathbb{R}_+$ to $\mathbb{R}$, so that $K(t)\star u_0^0$ is the solution to the Neumann problem for the heat equation with initial condition $u_0$. 

We look for results of convergence $u_0^a\to u_0^0$ as $a\to 0^+$. If we look at \eqref{eq:extension}, we immediately notice that, under the assumption \eqref{eq:assumpt}, there is a problem with uniform convergence: we have, for all $y>0$,
$$|{u}_0^a(-y)-u_0^0(-y)|\leq 2a \int_0^y e^{-a(y-\tilde{y})}|u_0(\tilde{y})|\,\mathrm{d}y. $$
However, the right hand side does not vanish as $a\to 0^+$ unless $u_0\in L^1(\mathbb{R}_+)$, which evades assumption \eqref{eq:assumpt01}. 

\begin{example}Let $u_0(y)=1$ for all $y\geq 0$. Then 
\begin{align*}
    u_0^0(y)=&1 \qquad \forall y \geq 0;\\ 
    u_0^a(y)=&\begin{cases} 1 & y\geq 0;\\ 
    -1+2e^{ay} & y<0
    \end{cases}
\end{align*}
As a result, $\sup_{y\in \mathbb{R}}|u_0^a(y)-u_0^0(y)|=2$ for all $a> 0$.
\end{example}
On the other hand, this limit is better behaved compared to $a\to+\infty$ in other aspects. Looking at \eqref{eq:derivatives3}, we notice that this time the terms at the boundary
$$2e^{-ay}\sum_{j=0}^{k-1}(-a)^{k-j}u_0^{(j)}(0) $$
all vanish uniformly in $y$ as $a\to 0^+$. However, we still need to require $u_0^{(j)}=0$ for all $j$, otherwise $u_0^a$ (and hence $u_0^a-u_0^0$) does not belong to $H^k$ for $k\geq 2$.

For this subsection, we then replace the assumption \eqref{eq:assumpt01} with the new assumption:
\begin{equation}\label{eq:assumpt02}\partial^k u_0(0)=0;\qquad  \partial^k u_0\in L^1(\mathbb{R}_+)\qquad \forall k\geq 0.\end{equation}
\begin{prop}Assuming \eqref{eq:assumpt02}, we deduce that for all $k\geq 0$,
$$\|u_0^a -u_0^0\|_{W^{k,\infty}}=O(a)\qquad \text{ as }a\to 0^+. $$
\end{prop}
\begin{proof}
From \eqref{eq:derivatives3} we deduce that for all $k\geq 0$,
\begin{align*}|(-1)^k\partial^k u_0^a(y)-\partial_k u_0^0(y)\|_{L^{\infty}}&\leq 2e^{-ay}\sum_{j=0}^{k-1}a^{k-j}|u_0^{(j)}(0)|+  2a\int_0^y e^{-a(y-\tilde{y})}|u_0^{(k)}(\tilde{y})|\,\mathrm{d}\tilde{y}\\
&\leq O(a)+ 2a \|u_0^{(k)}\|_{L^1}=O(a). \end{align*}
\end{proof} On the other hand, the assumption that all derivatives lie in $L^1(\mathbb{R}_+)$ must be maintained.

If we integrate by parts \eqref{eq:derivatives3} one more time, we obtain for all $k\geq 1$
\begin{align}\label{eq:derivatives4}
    (-1)^k\tilde{u}_0^{(k)}(-y)&=u_0^{(k)}(y)-2au_0^{(k-1)}(y)+ 2e^{-ay}\sum_{j=0}^{k-2} (-a)^{k-j}u_0^{(j)}(0)+2a^2\int_0^y e^{-a(y-\tilde{y})}u_0^{(k-1)}(\tilde{y})\,\mathrm{d}\tilde{y}.
\end{align}
This is also valid for $k=0$, provided that we define
$$u_0^{(-1)}(y):=\int_0^y u_0(z)\,\mathrm{d}z. $$

\begin{prop} Assuming \eqref{eq:assumpt02}, we deduce that for all $k\geq 1$,
$$\|\partial^ku_0^a-\partial^ku_0^0\|_{L^1}=O(a)\qquad \text{ as }a \to 0^+. $$
\end{prop}
\begin{proof}
From \eqref{eq:derivatives4}, we obtain 
\begin{align*}
    \|\partial^k u^a_0-\partial^k u^{0}_0\|_{L^1}&\leq O(a)+2a^2 \int_0^{\infty}\int_0^y e^{-a(y-\tilde{y})}|\partial^{k-1}u_0(\tilde{y})|\,\mathrm{d}\tilde{y}\,\mathrm{d}y\\ 
&=O(a)+2a^2\int_0^{\infty}e^{a\tilde{y}}\int_{\tilde{y}}^{\infty} e^{-ay}\,\mathrm{d}y\,|\partial^{k-1}u_0(\tilde{y})|\,\mathrm{d}\tilde{y}\\ 
    &=2a\|\partial^{k-1}u_0\|_{L^1}=O(a).
\end{align*}
\end{proof}
\begin{remark}If we further assume that
\begin{equation}\label{eq:assumpt3} \qquad y\mapsto \int_0^y u_0(z)\,\mathrm{d}z\in L^1(\mathbb{R}), \end{equation}
then the above result can also be extended to $k=0$.
\end{remark}

\begin{prop}Assuming \eqref{eq:assumpt02}, we deduce that for all $k\geq 0$,
$$\|\partial^ku_0^a-\partial^ku_0^0\|_{L^2}=O(a^{1/2})\qquad \text{ as }a \to 0^+. $$
\end{prop}
\begin{proof}We have, by Young's inequality for convolutions:
\begin{align*}
    \|\partial^ku_0^a-\partial^ku_0^0\|_{L^2}=2a\|e^{-a\tilde{y}}\star_{[0,y]} \partial^k u_0^0 \|_{L^2}\leq 2a\|e^{-ay}\|_{L^2}\|\partial^k u_0^0\|_{L^1}=\sqrt{2a}\|\partial^k u_0^0\|_{L^1}=O(a^{1/2}).
\end{align*}
\end{proof}
\begin{corollary}\label{convneu}Assume \eqref{eq:assumpt02} and \eqref{eq:assumpt3}. Then for all $k\geq 0$, $p\in [1,\infty)$ we have
$$\|\partial^ku^a -\partial^ku^0\|_{L^p}=O(a)\qquad \text{ as }a\to 0^+.$$
Assuming only \eqref{eq:assumpt02}, we can still deduce that for all $k\geq 0$,
\begin{align*}\| u^a-u^0\|_{H^k}=O(a^{1/2})\\
\| u^a-u^0\|_{W^{k,\infty}}=O(a).
\end{align*}
\end{corollary}
\begin{example}Suppose $u_0(y)=e^{-y}$, which is in $L^1(\mathbb{R}_+)$ with all its derivatives but does not satisfy \eqref{eq:assumpt3}. Then
$$u_0^a(y)=\begin{cases}e^{-y} & y \geq 0; \\ 
e^{y}-\frac{2a}{a-1}(e^y-e^{ay}) & y <0.
\end{cases} $$
Then as $a\to 0^+$,
\begin{align*}\int_{-\infty}^{\infty} (\partial^ku_0^a(y)-\partial^ku_0^0(y))\,\mathrm{d}y&= \frac{2(-1)^ka}{a-1}\l 1- a^{k-1} \r
\to  \begin{cases} 2 & k=0;\\ 
0& k>0.\end{cases}\end{align*}
In other words there is no $L^1$ convergence, although there is $L^1$ convergence of all derivatives. If we consider the $L^2$ norms:
$$\l\int_{-\infty}^{\infty}|u_0^a(y)-u_0^0(y)|^2\,\mathrm{d}y\r^{1/2}=\frac{2a}{|a-1|} \l \frac{\frac{1}{2}(a-1)^2}{a(a+1)}\r^{1/2}=\sqrt{2a}.$$
This shows that the order of convergence of $a^{1/2}$ from Corollary \ref{convneu} is optimal. To be precise, we need to multiply $u_0(y)=e^{-y}$ by a smooth cut-off function $\chi $ with $\chi^{(k)}(0)=0$ for all $k\geq 0$ and $\chi(y)=1$ for all $y\geq 1$. Then, assumption \eqref{eq:assumpt02} is fulfilled and the same estimates hold.
\end{example}
\section{A class of linearly unstable shear flows}\label{unstableflows}
The goal of this section is to prove that there exists an ample class of shear flows $U_s$ satisfying either assumption \eqref{eq:assumpt01} or \eqref{eq:assumpt02} which are linearly unstable for the Euler equation in the sense of Definition \ref{defunst}. We rely on a sufficient condition for linear instability is given by Z. Lin in \cite{zlin}. We state his main result in the half-line case. First of all, for a shear flow $U_s\in C^2(\mathbb{R}_+)$ admitting an inflexion point $y_0$, define the \textbf{inflexion value} $U_0:= U_s(y_0)$ and the function 
$$K(y):=\frac{-U_s''(y)}{U_s(y)-U_0}. $$
We say that $U_s$ is in class $\mathcal{K}^+$ if $K$ is a bounded and strictly positive function on $(0,+\infty)$. We stress that it is not necessary that $K(0)>0$. This is important because a profile satisfying \eqref{eq:assumpt01} or \eqref{eq:assumpt02} will necessarily have $K(0)=0$.

It is not difficult to construct profiles in the class $\mc{K}^+$. In essence, all that is required is that if the value of $U_s$ increases after an inflexion point, then the second derivative $U_s''$ must be negative until the next inflexion point is reached if it exists, and vice versa. A simple example is the function $\sin y$ or $\cos y$. 

Notice that if $U_s\in \mathcal{K}^+$  then, while it can admit multiple inflexion points, the inflexion value $U_0$ must be unique. Indeed by definition the sign of $U_s''$ must change at any inflexion point $\bar{y}$, hence the sign of $U_s-U_0$ must also change at $\bar{y}$ if we want $K>0$.

Secondly, if $\bar{y}$ is an inflexion point, then, assuming that $U_s\in C^3(\mathbb{R}_+)$ we have
$$K(\bar{y})=-\lim_{y\to \bar{y}}\frac{U_s'''(y)}{U_s'(y)}>0,$$
which implies that $U_s'(\bar{y})\neq 0$. If the sign of this derivative is positive, then $U_s(y)-U_0>0$ in a right neigbhorhood of $\bar{y}$, and therefore $U_s''<0$. If, instead, $U_s'(\bar{y})<0$, then $U_s(y)-U_0<0$ and $U_s''>0$ in a right neighborhood of $\bar{y}$. These signs cannot change until the next inflection point (if it exists at all).

The following result is a restatement of Theorem 1.2 from \cite{zlin} for shear profiles defined on the half-line. It is similar to Theorem 1.5(i) in \cite{zlin}, which covers the case of the full line. This version of the statement can be retrieved in \cite{paddick}, Theorem 4.2, however here we have explicitly removed the requirement that $K(0)>0$. The proof presented below focuses on clarifying why this requirement can be removed.
\begin{theorem}\label{zlin}Let $U_s\in C^2(\mathbb{R}_+)$, $U(y)\to U_{\infty}\in \mathbb{R}$ as $y\to +\infty$, and assume $U(y)$ takes the value $U_{\infty}$ at most a finite amount of times.   Suppose that $U_s\in \mc{K}^+$ and $\lim_{y\to +\infty}K(y)=0$. If the operator $-\partial_{yy}-K$ on $H_0^1(\mathbb{R}_+)\cap H^2(\mathbb{R}_+)$ has a strictly negative eigenvalue, then $U_s$ is linearly unstable for the Euler equation. 
\end{theorem}
\begin{proof}The proof is the same as Theorem 1.5(i) in \cite{zlin},  except that for each $n\in \mathbb{Z}_{\geq 0}$ we consider the interval $I_n= [n^{-1},n]$, on which $K$ is strictly positive, taking $n$ large enough so that $I_n$ contains an inflexion point of $U_s$. 
\end{proof}
The next result is a generalization of the argument used in \cite{paddick} to prove that the profiles
$$u_{\delta}(y)=\arctan (y-\delta)+ \zeta $$
for $\delta,\zeta \in \mathbb{R}$ satisfy the assumptions of \autoref{zlin}. Of course, the above family of profiles in unsuitable to us, as they do not satisfy \eqref{eq:assumpt01} or \eqref{eq:assumpt02}. 
\begin{prop}\label{exx}
Suppose that $U_s\in \mathcal{K}^+$ and $U_s$ satisfies assumption \eqref{eq:assumpt01} or \eqref{eq:assumpt02}. Let $U_{\infty}:=\lim_{y\to+\infty}U_s$ and suppose that $U_s$ takes the value $U_{\infty}$ at most a finite amount of times.
Then $U_s$ satisfies all the assumptions of \autoref{zlin}, and is therefore linearly unstable for the Euler equation.
\end{prop}
\begin{proof} By \eqref{eq:assumpt01} or \eqref{eq:assumpt02}, we know that all the derivatives of $U_s$ are integrable over $\mathbb{R}_+$ and vanish at infinity. 

The operator $-\partial_{yy}-K^2$ having a strictly negative eigenvalue is equivalent to the quadratic form 
 $$Q(\phi):=\int_{\mathbb{R}_+}\l |\phi'|^2 - K|\phi|^2\r,\qquad \phi \in H_0^1(\mathbb{R}_+),$$
 taking a negative value for some function $\phi$. We will construct such a function from the profile $U_s$.
 \begin{enumerate}
     \item First of all, define for all $\eta>0$
     $$U_{s,\eta}(y):= U_s(y+y_0-\eta)-U_0,\qquad y \in [\eta-y_0,+\infty). $$
     This implies that $U_{s,\eta}(y)$ has an inflexion point at $y=\eta$, and the inflexion value is $U_{s,\eta}(\eta)=0.$
     \item Next, define the functions
     $$w_{\eta}^n:=\begin{cases}
     0 & y \leq \eta \\ 
     U_{s,\eta}\chi (y/n) & y \geq \eta,
     \end{cases} $$
     where $\chi $ is a smooth cut-off function supported in $[0,2]$, with $\chi=1$ in $[0,1]$. Then $w_{\eta}^n\in H_0^1(\mathbb{R}_+) $. Since $K\in L^{\infty}(\mathbb{R}_+)$, and by \eqref{eq:assumpt01} or \eqref{eq:assumpt02} all the derivatives of $U_s$ are integrable over $\mathbb{R}_+$ and vanish at infinity, we have
     $$\lim_{n\to +\infty}Q(w_{\eta}^n)=\int_{\eta}^{+\infty}\l |U_{s,\eta}'|^2- K|U_{s,\eta}|^2 \r =:Q(\eta). $$
     Let us show that $Q(y_0)=0$. 
     \begin{align*}
         Q(y_0)&=\int_{y_0}^{\infty}\l |U_s'(y)|^2 + \frac{U_s''(y)}{U_s-U_0}|U_s(y)-U_0|^2 \r\,\mathrm{d}y\\ 
         &= \int_{y_0}^{\infty}\l |U_s'(y)|^2 + {U_s''(y)}(U_s(y)-U_0)\r\,\mathrm{d}y\\ 
         &= -U_0\int_{y_0}^{\infty}U_s''+ \ll U_sU_s'\rr_{y_0}^{\infty}\\
         &=(-U_0+U_{\infty})\lim_{y\to +\infty}U_s'(y)=0.
     \end{align*}
     If we compute the derivative of $Q$ as a real variable function, we obtain that for any $\eta>0$,
     $$Q'(\eta)= \int_{\eta}^{\infty}\l -2 U_{s,\eta}''U_{s,\eta}+ 2K U'_{s,\eta}U_{s,\eta}\r -|U_{s,\eta}'(\eta)|^2.$$
     In particular,
     $$Q'(y_0)=-4\int_{y_0}^{\infty}U_s'U_s''- |U_s'(y_0)|^2= -2 [(U_s')^2]_{y_0}^{\infty}-(U_s'(y_0))^2 =(U_s'(y_0))^2>0.$$
     
     \item We know that $Q(y_0)=0$ and $Q'(y_0)>0$. Therefore, for some $\eta_0 \in (0,y_0)$ we must have $Q(\eta_0)<0$. Hence, for some $n\in \mathbb{Z}_{\geq 0}$ we have $Q(w_{\eta_0}^n)<0$. This concludes the proof. 
 \end{enumerate}

\end{proof}
\begin{example}\label{gevrey}An explicit example of a flow satisfying the assumptions of Proposition \ref{exx} is given by
$$U_s(y)=e^{-y^{-\frac{1}{\rho-1}}},\qquad \rho>1. $$
 Indeed, its first three derivatives are
\begin{align*}
    U_s'(y)&=\frac{1}{\rho-1} y^{-\frac{\rho}{\rho-1}}e^{-y^{-\frac{1}{\rho-1}}}, \\
    \quad U_s''(y)&= \frac{1-\rho y^{\frac{1}{\rho-1}}  }{(\rho-1)^2y^{\frac{2\rho}{\rho-1}}}e^{-y^{-\frac{1}{\rho-1}}},\\
    \quad U_s'''(y)& = \frac{1+2\rho^2y^{\frac{2}{\rho-1}}-\rho y^{\frac{1}{\rho-1}}\l 3+ y^{\frac{1}{\rho-1}} \r}{(\rho-1)^3y^{\frac{3\rho}{\rho-1}}}e^{-y^{-\frac{1}{\rho-1}}},
\end{align*}
so that $U_s'>0$, there is a unique inflexion point at $y=y_0:=\rho^{1-\rho}$, with $U_s''>0$ for $y\in \l 0,y_0\r$ and $U_s''<0$ for $y>y_0$, and $U_s'''\l y_0\r = -\dfrac{\rho^{3\rho-1}}{e^{\rho}(\rho-1)^3}<0$. This implies that $U_s\in \mathcal{K}^+$. Indeed $K(y)>0$ for $y\neq 0,y_0$. Moreover, by De L'Hopital's rule we have
$$ \lim_{y\to y_0}K(y)=-\lim_{y\to y_0}\frac{U_s^{(3)}(y)}{U_s'(y)}=-\frac{U_s^{(3)}(y_0)}{U_s'(y_0)}>0.$$
 By \eqref{eq:assumpt}, we have $\lim_{y\to +\infty}U_s(y)=U_{\infty}\in \mathbb{R}$, and $\lim_{y\to +\infty}K(y)=0$. Since $U_s'>0$, it never actually takes the value $U_{\infty}$. One can argue by induction that all the derivatives are bounded, integrable and vanish at $y=0$. Thus, this is a linearly unstable shear flow satisfying assumption \eqref{eq:assumpt01}.
\end{example}

\begin{remark} The above flow belongs to the Gevrey class $G^{\rho}(\mathbb{R}_+)$.
\end{remark}
The flow from the previous example does not satisfy assumption \eqref{eq:assumpt02}, as $U_{\infty}\neq 0$. To construct a linearly unstable flow with $U_s(0)=U_{\infty}=0$, we need at least two inflexion points $y_1,y_2$, with $U_s(y_1)=U_s(y_2)=U_0$. One can then easily construct a smooth profile $U_s$ with $U_s>0$ on $(0,\infty)$ satisfying all the requirements by requiring the following:
\begin{enumerate}
    \item $U_s'>0,U_s''>0$ on $(0,y_1)$;
    \item $U_s''<0$ on $(y_1,y_2)$, and $U_s'$ changes its sign somewhere in $(y_1,y_2)$,
    \item $U_s'<0,U_s''>0$ on $(y_2,+\infty)$.
\end{enumerate}

\begin{figure}[t]
\includegraphics[scale=0.5]{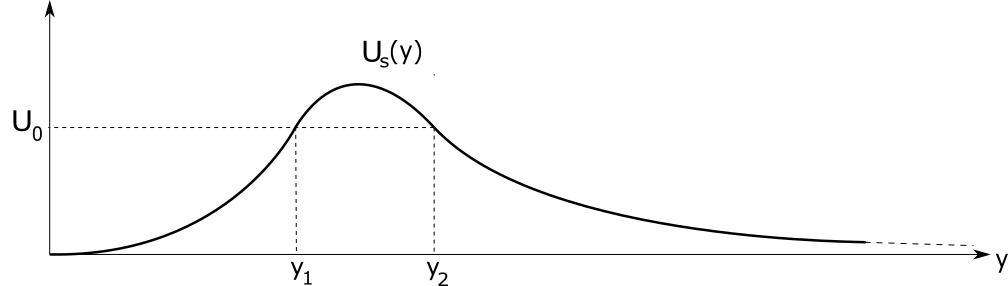}
\caption{A linearly unstable shear profile satisfying assumption \eqref{eq:assumpt02}.}
    \label{fig:profile}
\end{figure}
\newpage 
\section{Grenier's instability with a viscosity-dependent Navier boundary condition}\label{main0}
Let $\gamma \in \mathbb{R}$. Let $U_s:\mathbb{R}_+\to \mathbb{R}$ be a smooth shear flow, linearly unstable for the Euler equation in the sense of Definition \ref{defunst}.
 Consider the following Navier-Stokes equations with the Navier boundary condition:
 \begin{equation}\label{eq:nsnavier}\begin{cases}\partial_t\mbf{u}^{\nu}+\mbf{u}^{\nu}\cdot \nabla \mbf{u}^{\nu}+\nabla p^{\nu}=\nu \Delta \mbf{u}^{\nu};\\
\nabla \cdot \mbf{u}^{\nu}=0;\\
\partial_y u^{\nu} = \nu^{-\gamma}u^{\nu}& \text{ at }y=0;\\
v^{\nu}=0& \text{ at }y=0.\end{cases}
 \end{equation}
 
 Let $\tilde{y}=y/\sqrt{\nu}$, and let $u_s^{\nu}=u_s^{\nu}(t,\tilde{y})$ be the solution to the heat equation
\begin{equation}\label{eq:heatrobin}
\begin{cases}\partial_t u_s^{\nu}(t,\tilde{y})=\partial_{\tilde{y}\tilde{y}}u_s^{\nu}(t,\tilde{y}) & (t,\tilde{y})\in \mathbb{R}_+\times \mathbb{R}_+;\\
\partial_Y u_s^{\nu}(t,0)=\nu^{1/2-\gamma} u_s^{\nu}(t,0) & t\in \mathbb{R}_+;\\
u_s^{\nu}(0,\tilde{y})= U_s(\tilde{y}) & \tilde{y}\in \mathbb{R}_+.
\end{cases}
\end{equation}
We will also use $u_s^{\nu}$ to denote the shear flow $(u_s^{\nu},0)$ which is therefore a solution to \eqref{eq:nsnavier} in the original variables $(t,x,y)$. 

If we take the limit as $\nu\to 0$ in \eqref{eq:heatrobin}, we expect convergence of $u_s^{\nu}$ to the solution of the Dirichlet or Neumann problem for the heat equation, respectively if $\gamma >1/2$ or $\gamma <1/2$. In order to establish our main result, we will need the convergence results Corollary \ref{conv} and Corollary \ref{convneu} respectively. For those to hold, we must require the following assumption on the profile $U_s$, depending on the sign of $\gamma-1/2$:
\begin{equation}\label{eq:assumpt} \lim_{y\to +\infty}U_s(y)=U_{\infty}\in \mathbb{R},\qquad  U_s^{(k)}(0)=0\;\forall k\in \mathbb{Z}_{\geq 0}\qquad \textup{ and }\qquad \begin{cases}U_s^{(k)}\in L^1(\mathbb{R}_+)\;\forall k\geq 1 & \text{ if }\gamma > 1/2; \\ 
U_s^{(k)}\in L^1(\mathbb{R}_+)\;\forall k \geq 0 & \text{ if }\gamma < 1/2.\end{cases}
\end{equation}
Notice that under this assumption, since all the derivatives of $U_s$ vanish at the origin, $U_s$ satisfies the compatibility conditions of \eqref{eq:heatrobin} for all orders, and for all $\nu>0$. Thus $u_s^{\nu}$ is smooth up to the boundary.

In Section \ref{unstableflows} we confirmed the existence of profiles satisfying \eqref{eq:assumpt} which are linearly unstable for the Euler equation. These profiles cannot be analytic, but they can be found in the Gevrey classes $G^{\rho}$ for any $\rho>1$ (see Example \ref{gevrey}). When $\rho\leq 2$, for these flows the Prandtl equation is well-posed (see \cite{dgv}). Hence, by the next result, instability of the boundary layer expansion for the Navier boundary condition can occur even when the Prandtl equation is well-posed, in line with the no-slip case.

Finally, we remark that as the case $\gamma=1/2$ has already been treated in \cite{paddick}, throughout this paper we will focus on the case $\gamma\neq 1/2$.

We are now ready to state the main result of this paper.
\begin{theorem}\label{main} For $\nu >0$, let $u_s^{\nu}=(u_s^{\nu},0)\in C^{\infty}(\mathbb{R}_+)$ be a family of shear flows defined by \eqref{eq:heatrobin}. Then for any $N\in \mathbb{Z}_{\geq 1}$ there exists a family of solutions $\mbf{u}^{\nu}=(u^{\nu},v^{\nu})$ to \eqref{eq:nsnavier},
constants $C,\delta>0$ and times $\tilde{T}^{\nu}\searrow 0$ such that for all $\nu>0$,
\begin{align}\label{eq:main0}\|\mathbf{u}^{\nu}(0,x,y)-u_s^{\nu}(0,y/\sqrt{\nu})\|_{L^{\infty}} &\leq C\nu^N\\  \label{eq:main1}
\|\mbf{u}^{\nu}(\tilde{T}^{\nu},x,y)-u_s^{\nu}(\tilde{T}^{\nu},y/\sqrt{\nu})\|_{L^{\infty}}&\geq \delta\nu^{\theta},
\end{align}
where $\theta$ is a continuous and increasing function of $\gamma$ given by
\begin{equation}\label{eq:thetadef}\theta:=\begin{cases} \frac{1}{4} & \gamma \geq \frac{3}{4};\\ 
\gamma - \frac{1}{2} & \frac{1}{2}<\gamma < \frac{3}{4};\\
0 & \gamma \leq \frac{1}{2}.
\end{cases}\end{equation}
Moreover, for all $s>2\theta +1$, we have
\begin{equation}\label{eq:main2}\|\mbf{u}^{\nu}(\tilde{T}^{\nu},x,y)-u_s^{\nu}(\tilde{T}^{\nu},y/\sqrt{\nu})\|_{\dot{H}^s}\to +\infty.\end{equation}
\end{theorem}
Notice that $(t,x,y)\mapsto u_s^{\nu}(t,y/\sqrt{\nu})$ satisfies \eqref{eq:nsnavier}. This shows that the Navier-Stokes equations \eqref{eq:nsnavier} are unstable around the family of shear flows $u_s^{\nu}$.
When $\gamma\geq 1/2$, the above result also shows the instability of the boundary layer expansion
 \begin{equation}\label{eq:blexpnew}\mbf{u}^{\nu}(t,x,y)\sim  \mbf{u}^E(t,x,y)+ \mbf{u}^b\l t,x,\frac{y}{\sqrt{\nu}}\r\qquad \text{ as }\nu \to 0, \end{equation}
 with $\lim_{\tilde{y}\to +\infty}\mbf{u}^b(t,x,\tilde{y})=0$. Notice that such a boundary layer expansion is necessarily unique. Indeed, suppose there was another expansion for $\mbf{u}^{\nu}$ given by $\tilde{\mbf{u}}^E,\tilde{\mbf{u}}^b$, so that
 $$\mbf{u}^E(t,x,y)+ \mbf{u}^b\l t,x,\frac{y}{\sqrt{\nu}}\r\sim \tilde{\mbf{u}}^E(t,x,y)+\tilde{ \mbf{u}}^b\l t,x,\frac{y}{\sqrt{\nu}}\r. $$ 
 Taking the limit $\nu\to 0$ pointwise, we obtain $\mbf{u}^E=\tilde{\mbf{u}}^E$, and hence $\mbf{u}^b=\tilde{\mbf{u}}^b$. In this case, the inequality \eqref{eq:main0} tells us that at time $t=0$ the boundary layer expansion \eqref{eq:blexpnew} of $\mbf{u}^{\nu}$ and $u_s^{\nu}$ must coincide.  However,  \eqref{eq:main1} tells us that at time $\tilde{T}^{\nu}$ the boundary layer expansions diverge by at least $O(\nu^{\theta})$.

On the other hand, when $\gamma <1/2$, \autoref{main} does not give us any information about boundary layer expansions. Indeed, the expected boundary layer expansion would be 
$$\mbf{u}^{\nu}(t,x,y)\sim \mbf{u}^E(t,x,y)+\nu^{\min\lll 1/2-\gamma; 1/2\rrr}\mbf{u}^b\l t,x,\frac{y}{\sqrt{\nu}}\r. $$
However, the shear flows $u_s^{\nu}$ cannot satisfy the above formula at $t=0$, as $\left.u_s^{\nu}\right|_{t=0}=U_s(y/\sqrt{\nu})$, which appears with a coefficient of order one with respect to the viscosity.  This is also the reason why our result is not in contradiction with the result of boundary layer expansion validity by Iftimie and Sueur \cite{iftimie} when $\gamma=0$ (viscosity-independent slip length), or with the result of convergence of Navier-Stokes to Euler by Paddick \cite{paddick} which establishes convergence of order $\nu^{\frac{1-\gamma}{2}}$ for all $\gamma <1$. 
 
\subsection{General strategy} 
As the case $\gamma=1/2$ has already been treated in \cite{paddick}, we will focus our proof on the two cases $\gamma >1/2$ and $\gamma <1/2$. As in \cite{grenier} and \cite{paddick}, we start by applying the isotropic scaling 
$$(t,x,y)\mapsto \l \frac{t}{\sqrt{\nu}},\frac{x}{\sqrt{\nu}},\frac{y}{\sqrt{\nu}}\r. $$
From now on, we will work exclusively with the new variables, which we will still denote with $(t,x,y)$. After the scaling, \eqref{eq:nsnavier} is transformed into the following: 
\begin{equation}\label{eq:nsnavier2}\begin{cases}\partial_t\mbf{u}^{\nu}+\mbf{u}^{\nu}\cdot \nabla \mbf{u}^{\nu}+\nabla p^{\nu}=\sqrt{\nu} \Delta \mbf{u}^{\nu};\\
\nabla \cdot \mbf{u}^{\nu}=0;\\\partial_y u^{\nu} = {\nu}^{1/2-\gamma}u^{\nu}& \text{ at }y=0;\\
v^{\nu}=0& \text{ at }y=0.\end{cases}
 \end{equation}
Notice that the Navier-Stokes equations are preserved, except the new viscosity is $\sqrt{\nu}$, while in the boundary condition, $\gamma$ becomes $\gamma -1/2>0$ (or $2\gamma-1 $ with respect to $\sqrt{\nu}$). For instance, the exponent $\gamma=1/2$, considered by Paddick, is trasformed into $\gamma=0$. The new boundary layer expansion for $\mbf{u}^{\nu}$, as per \eqref{eq:expgen}, becomes
\begin{equation}\label{eq:blexpiso}
    \mbf{u}^{\nu}(t,x,y)\sim\mbf{u}^E(t,x,y)+ \nu^{a}\mbf{u}^b\l t,x,Y\r,
\end{equation}
where $Y:=y/\nu^{1/4}$, and $a$ is a non-negative number representing the amplitude of the boundary layer, defined as
 \begin{equation}\label{eq:adef}a:=\frac{1}{4}-\theta=\begin{cases}0 & \gamma \geq \frac{3}{4};\\ 
 \frac{3}{4}-\gamma & \frac{1}{2}<\gamma <\frac{3}{4};\\
 \frac{1}{4}& \gamma \leq \frac{1}{2};
 \end{cases}.\end{equation}
Thus the critical exponent becomes $\gamma=3/4$.  The Euler equation is invariated, as are $L^{\infty}$ norms, whereas spatial $L^2$ norms are increased by a factor of $\nu^{-1/2}$. The shear flows are written as $u_s=u_s\l \sqrt{\nu}t, y\r$ in the new coordinates, which means that their dependence from $\nu$ is now smooth.

 An \textbf{approximate solution} to \eqref{eq:nsnavier2} is a solution up to some error function $\mbf{R}^{\textup{app}}$, which can be made arbitrarily small. Taking inspiration from the boundary layer expansion \eqref{eq:blexpiso}, we will construct our approximate solution as 
\begin{equation}\label{eq:uapp}(u^{\textup{app}},v^{\textup{app}})(t,x,y)= (u_s^{\nu},0)(\sqrt{\nu}t,y)+ (u^I,v^I)(t,x,y) + \nu^a (u^b,\nu^{1/4}v^b)\l t,x,Y\r, \end{equation}
 where

In first order, $u_s^{\nu}(\sqrt{\nu}t)+\mbf{u}^I(t)$ will therefore satisfy the Euler equations, while $\mbf{u}^b=(u^b(t,x,Y),\nu^{1/4}v^b(t,x,Y))$ will satisfy a Stokes equation. The boundary conditions will be chosen appropriately so that $\mbf{u}^{\textup{app}}(t,x,y)$ will satisfy the Navier boundary condition up to a small error:
 $$\begin{cases}\partial_y u^{\textup{app}}=\nu^{1/2-\gamma}u^{\textup{app}}+r_1^{\textup{app}};\\
 v^{\textup{app}}=r_2^{\textup{app}};
 \end{cases}\qquad \text{ at }y=0. $$

A heuristic justification for the value of $a$ can be given as follows. Let $u^E(t):=u^I(t)+u_s^{\nu}(\sqrt{\nu}t)$. 
 Plugging the ansatz \eqref{eq:uapp} into \eqref{eq:nsnavier2}$_3$ and multipying by $\nu^{\gamma-1/2}$, we obtain 
 $$\nu^{\gamma-1/2}\partial_y u^E+ \nu^{\gamma+a-3/4}\partial_Y u^b-\nu^a u^b= u^E, $$
 up to a small error which can be assumed to be smaller than all the other terms. As $\nu \to 0$, for the above relation to hold, the left-hand side must be asymptotic to a constant. But this is only possible if at least one between $\gamma+a-3/4$ and $a$ is zero, and the other is greater or equal to zero, which leads to $a$ being defined as in \eqref{eq:adef}.

Fix $\delta>0$. We first construct an approximate solution to \eqref{eq:nsnavier2} $\mbf{u}^{\textup{app}}$, a bounded subset $\Omega_A\subset \mathbb{R}^2_+$ and times $T^{\nu}$ with $\tilde{T}^{\nu}:=\sqrt{\nu}T^{\nu}\to 0$ such that
\begin{align}
\label{eq:idk13}\|\left.\l \mbf{u}^{\textup{app}}-u_s^{\nu}\r\right|_{t=0}\|_{L^2}&\leq C \nu^N;\\ \label{eq:idk14} \|\mbf{u}^{\textup{app}}(T^{\nu})-u_s^{\nu}(\sqrt{\nu}T^{\nu})\|_{L^2(\Omega_A)}&\geq 2\delta\nu^{\theta}. \end{align}
Starting from $\mbf{u}^{\textup{app}}$, we then construct an exact solution $\mbf{u}^{\nu}$ to \eqref{eq:nsnavier2} such that
\begin{align}
\label{eq:idk11}\mbf{u}^{\nu}(0)&= \mbf{u}^{\textup{app}}(0);\\
 \label{eq:idk12}\|\left. \l \mbf{u}^{\nu}-\mbf{u}^{\textup{app}}\r\right|_{t=T^{\nu}}\|_{L^2}&\leq \delta \nu^{\theta};
 \end{align}
where $\theta$ is defined as in \eqref{eq:thetadef}. Once we have these, then 
\begin{align*}
    \|\mbf{u}^{\nu}(T^{\nu})-u_s^{\nu}(\sqrt{\nu}T^{\nu})\|_{L^{\infty}}&\geq \|\mbf{u}^{\nu}(T^{\nu})-u_s^{\nu}(\sqrt{\nu}T^{\nu})\|_{L^{\infty}(\Omega_A)}\\
    &\gtrsim \|\mbf{u}^{\textup{app}}(T^{\nu})-u_s^{\nu}(\sqrt{\nu}T^{\nu})\|_{L^2(\Omega_A)}-\|\mbf{u}^{\nu}(T^{\nu})-\mbf{u}^{\textup{app}}(T^{\nu})\|_{L^2}\\ 
    &\geq 2\delta\nu^{\theta}- \delta\nu^{\theta}= \delta\nu^{\theta}.
\end{align*}

The estimates at time $t=0$, \eqref{eq:idk11} and \eqref{eq:idk13}, will hold by construction.  The estimate \eqref{eq:idk12} will be deduced from energy estimates, as in Section \ref{energyestt}, and \eqref{eq:idk14} will follow by the construction of $\mbf{u}^{\textup{app}}$. Note that \eqref{eq:idk14} and \eqref{eq:idk12} do not imply the instability of \autoref{main} in the $L^2$ norms, as after scaling back to the original variables, we would lose a $\sqrt{\nu}$ factor. 

\subsection{Structure of the approximate solution}

We want to construct $\mbf{u}^{\nu}$ starting from an approximate solution $\mbf{u}^{\textup{app}}$. This is built according to \eqref{eq:uapp},
where $\mbf{u}^I=(u^I,v^I)$ is constructed so that $u_s^{\nu}+\mbf{u}^I$ satisfies the Navier-Stokes equations with an error $\mbf{R}^{\textup{app}}$ and the slip boundary condition, and $ \mbf{u}^b=(u^b,\nu^{1/4} v^b)$ corrects the boundary condition. Ultimately, $\mbf{u}^{\textup{app}}(t)-u_s^{\nu}(\sqrt{\nu}t)$ will satisfy the Navier-Stokes equations \eqref{eq:nsnavier2} with the Navier boundary condition  up to a small error $\mbf{r}^{\textup{app}}$.

The standard procedure for the construction is to expand the terms $\mbf{u}^I$ and $\mbf{u}^b$ as power sums with respect to the viscosity. Let $n\in \mathbb{Z}$, $n\geq 2$ be such that
\begin{equation}\label{eq:ndef}\begin{cases}
2^{-n}\leq \gamma- \frac{3}{4} &\text{ if } \gamma \geq \frac{3}{4};\\ 
2^{-n}\leq \gamma-\frac{1}{2} &\text{ if } \frac{1}{2}< \gamma < \frac{3}{4};\\
2^{-n}\leq \frac{1}{2}\l \frac{1}{2}-\gamma \r &\text{ if } \gamma < \frac{1}{2}.
\end{cases}\end{equation}
Note that in each case, $2^{-n}\leq \theta$.

Denote $\mbf{w}(t):=\mbf{u}^{\textup{app}}(t)-{u}_s^{\nu}(\sqrt{\nu}t)$. Constructing $u^{\textup{app}}$ is then equivalent to constructing $\mbf{w}$, which must satisfy the equations
\begin{equation}\label{eq:idk2}\begin{cases}\partial_t \mbf{w} + (U_s\cdot \nabla) \mbf{w} + ( \mbf{w}\cdot \nabla) U_s+( \mbf{w}\cdot \nabla)  \mbf{w}+ \nabla p = \sqrt{\nu}\Delta  \mbf{w}+ \nu^{2^{-n}}S \mbf{w};\\ 
\nabla \cdot  \mbf{w}=0;\\
\partial_y  \mbf{w}\cdot \tau = \nu^{1/2-\gamma} \mbf{w}\cdot \tau;\\ 
 \mbf{w}\cdot n =0.
\end{cases} \end{equation}
where
$$ S \mbf{w}(t):= \frac{U_s- u_s^{\nu}(\sqrt{\nu}t)}{\nu^{2^{-n}}}\cdot \nabla  \mbf{w}(t) +  \mbf{w}(t)\cdot \nabla \l \frac{U_s- u_s^{\nu}(\sqrt{\nu}t)}{\nu^{2^{-n}}} \r. $$
The reason behind this definition is that as we will prove in Lemma \ref{usbound}, $S \mbf{w} = O(1)$ as $\nu \to 0$. 

We are going to implement the ansatz \eqref{eq:uapp} for the construction of $\mbf{u}^{\textup{app}}$. The term $\mbf{u}^I$ will be constructed so that it solves \eqref{eq:idk2}, but without the Navier boundary condition \eqref{eq:idk2}$_3$. This way, \eqref{eq:idk2} reduces to a linearized Euler equation around $U_s$ in first order approximation. The term $\mbf{u}^b$ will be constructed so that $\mbf{u}^I+\nu^a \mbf{u}^b$ fully solves \eqref{eq:idk2}, correcting the boundary condition. Since $\mbf{u}^b$ is a function of $\l t, x, y/\nu^{1/4}\r$, it will satisfy a Stokes equation with a Dirichlet, Neumann or Robin boundary condition, depending on the value of $\gamma$. 

Of course, the approximate solution $\mbf{u}^{\textup{app}}$ we construct also needs to satisfy \eqref{eq:idk14}. In order to achieve this, in general, the function $\mbf{w}$ will not satisfy \eqref{eq:idk2} exactly but will leave a remainder $\mbf{R}^{\textup{app}}$, which needs to be small enough so that \eqref{eq:idk12} holds by energy estimates. 

Choose $M\in \mathbb{N}$, which may be arbitrarily high. Since we want $\left.(\mbf{u}^{\textup{app}}-u_s^{\nu})\right|_{t=0}=O(\nu^N)$ as $\nu\to 0$, we will use the following ansatz: 
\begin{equation}\label{eq:ansatz2}
\mbf{u}^I=\nu^N \sum_{j=0}^M \nu^{j2^{-n}}\mbf{u}_j^I;\qquad \mbf{u}^b = \nu^N \sum_{j=0}^M \nu^{j2^{-n}} \mbf{u}_j^b.\end{equation}
This, by the definition of $n$, ensures that the terms of order $\nu^{\gamma-1/2}$ and $\nu^{\gamma-3/4}$ can be moved to a higher order, so that they only appear in the remainder of the respective equations. Recall that $\mbf{u}^b$ is then multiplied by a factor ${\nu}^a$ where $a= \max \lll 3/4-\gamma;0 \rrr$. 

The construction of $\mbf{w}$ will start from $\mbf{u}^I_0$, which exhibits the required estimates, but does not solve the  required equation. However, from $\mbf{u}^I_0$ we can construct an approximate solution $\mbf{u}^I$ of  \eqref{eq:idk2} by adding lower order terms in $\nu$. On top of that we need a boundary term $\mbf{u}^b$, which corrects the boundary condition so that the Navier boundary condition is satisfied. Such a term is a function of $\l t,x,y/\nu^{1/4}\r$, therefore its $L^2$ norm scales as $\sim \nu^{1/8}$. As a result, \eqref{eq:idk14} will hold:
\begin{align*}&\| \mbf{u}^{\textup{app}}(T^{\nu})-u_s(\sqrt{\nu}T^{\nu})\|_{L^{2}(\Omega_A)}\geq \\ 
&\geq\nu^N\l  \|\mbf{u}_0^I\|_{L^2(A)}-\sum_{j=1}^M \nu^{j2^{-n}}\|\mbf{u}_j^I\|_{L^2}-\sum_{j=0}^M\nu^{j2^{-n}}\|\mbf{u}^b_j\|_{L^2(y)} \r\\ 
&\geq \nu^N\l \|\mbf{u}_0^I\|_{L^2(A)}-\nu^{2^{-n}}\sum_{j=1}^M \nu^{(j-1)2^{-n}}\|\mbf{u}_j^I\|_{L^2}-\nu^{1/8}\sum_{j=0}^M\nu^{j2^{-n}}\|\mbf{u}_j^b\|_{L^2(Y)}\r\\
&\gtrsim \frac{\nu^N}{2}\|\mbf{u}_0^I\|_{L^2(A)}\qquad \qquad\text{ as }\nu\to 0.
\end{align*}

We can now study the equations satisfied by $\mbf{u}_j^I$ and $\mbf{u}_j^b$. Each of the equations that follow should be paired with an initial condition $\mbf{u}_j^{*}(x,y)$, $^*=I,b$. Because we are constructing an approximate solution, we can choose the initial condition arbitrarily, as long as it satisfies the following conditions:
\begin{itemize}
    \item it is compatible with the boundary condition and remainder;
\item together with its derivatives, it is in $L^2$ and decays exponentially as $Y\to +\infty$: in other words,
\begin{equation}\label{eq:icbound}|\partial_x^{\ell}\partial_Y^k\mbf{u}_j^*(x,y)|\leq |g^*_{j,k,\ell}(x)|e^{-\lambda_j y}, \qquad ^*=I,b,\end{equation}
for some $\lambda>0, g^*_{j,k,\ell}\in L^2(\mathbb{R})$.
\end{itemize}
The equations satisfied by $\mbf{u}_j^I$ involve all the terms of order between $N+j2^{-n}$ and $N+(j+1)2^{-n}$, including the former but not the latter. We obtain the following inhomogeneous linearized Euler equations:
 \begin{equation}\label{eq:internal} \begin{cases}
 \partial_t \mbf{u}_j^I + U_s\cdot \nabla \mbf{u}_j^I + \mbf{u}_j^I \cdot \nabla U_s + \nabla p_j^I = \mbf{R}_j^I;\\ 
 \nabla \cdot \mbf{u}_j^I=0;\\
 v_j^I = -\nu^{a+1/4-2^{-n}}v_{j-1}^b;\\ 
 \end{cases}
 \end{equation}
 where
 \begin{equation}\label{eq:intremj}\mbf{R}_j^I = S\mbf{u}^I_{j-1}+ \Delta \mbf{u}^I_{j-2^{n-1}}+ \sum_{ j_1+j_2=j-2^nN}\mbf{u}^I_{j_1}\cdot \nabla \mbf{u}^I_{j_2}.\end{equation}
 Notice that $a+1/4-2^{-n}\geq a\geq 0$, and $0\leq \theta-2^{-n}<2^{-n}$. Thus $u_s+\mbf{u}^I$ satisfies the Navier-Stokes equations with the slip boundary condition \eqref{eq:internal}$_3$ and an error $\mbf{R}^I$ consisting of all the terms of order greater or equal to $N+(M+1)2^{-n}$. Therefore,
\begin{equation}\label{eq:intfinal}\mbf{R}^I =\nu^N \sum_{j\geq M+1}\nu^{j2^{-n}}\mbf{R}^I_j.\end{equation}
Notice that the above sum is actually finite, as $\mbf{R}_j^I=0$ for $j>M+2^nN$. 
 
The equations satisfied by $\mbf{u}_j^b=(u_j^b,\nu^{1/4}v_j^b)$ involve all the terms where $u_j^b$ is of order between $N+a+j2^{-n}$ and $N+a+(j+1)2^{-n}$.
 \begin{equation}\label{eq:bdry}\begin{cases}\partial_t \mbf{u}_j^b -\partial_{YY}\mbf{u}_j^b +\nabla_{x,Y} p_j^b = \mbf{R}^b_j;\\ 
 \nabla_{x,Y} \cdot \mbf{u}_j^b = 0;\\ 
 \textup{B.C.}(u^b_j,\gamma);\\
 \lim_{Y\to +\infty}v^b_j=0.
 \end{cases} \end{equation}
 where B.C.($u^b_j,\gamma$) is the appropriate boundary condition on $u^b_j$, depending on $\gamma$, which we will derive in the next subsection, and
 \begin{align}\label{eq:bdryj}\mbf{R}^b_j&=\ll\frac{u_s^{\nu}}{\nu^{2^{-n}}}\cdot \nabla \mbf{u}^b_{j-1}+ \mbf{u}^b_{j-1}\cdot \nabla \l \frac{u_s^{\nu}}{\nu^{2^{-n}}}\r\rr-\partial_{xx}\mbf{u}_b^{j-2^{n-1}}+\nu^a\sum_{k+\ell = j-2^nN}\mbf{u}_k^b\cdot \nabla \mbf{u}_{\ell}^b \nonumber\\ 
 &+\sum_{j_1+j_2=j-2^nN}\l u^I_{j_1}\partial_x\mbf{u}_{j_2}^b + u_{j_1}^I\partial_y \mbf{u}_{j_2+2^{n-2}}^b+\mbf{u}_{j_2}^b \cdot \nabla \mbf{u}^I_{j_1} \r \end{align}
First notice that all the terms in the above equation appear with index strictly smaller than $j$. A special remark must be given for the term containing $\partial_y \mbf{u}_{j_2+2^{n-2}}^b$. The extra $2^{n-2}$ indices arise to compensate for the differentiation by $y$ causing a loss of a $\nu^{1/4}$ factor. However, as $j_2 \leq j-2^nN$, we know that
 $$j_2+2^{n-2}\leq j+2^{n-2}(1-4N)<j,  $$
 which holds since $N\geq 1$.
 
 Secondly, $v^b_j$ can be derived from $u^b_j$ using the divergence free condition \eqref{eq:bdry}$_2$:
 \begin{equation}\label{eq:divfree}v^b_j(t,x,Y)=-\int_Y^{+\infty}\partial_x u^b_j(t,x,Z)\,\mathrm{d}Z; \end{equation}
 using this expression, the condition at infinity \eqref{eq:bdry}$_4$ is automatically satisfied.
 
 Thus $u_s+\mbf{u}^I+\nu^a \mbf{u}^b$ satisfies the Navier-Stokes equations with error $\mbf{R}^{\textup{app}}=\mbf{R}^I+\nu^a \mbf{R}^b$, where $\mbf{R}^b$ consists of all the terms of order greater or equal to $N+a+(M+1)2^{-n}$:
 \begin{equation}\label{eq:bdryfinal}\mbf{R}^{b}= \nu^N\sum_{j\geq M+1}\nu^{j2^{-n}}\mbf{R}^b_j.\end{equation}
As with $\mbf{R}^I$, the sum is actually finite.

  As in Proposition \ref{heatestimate}, we can then estimate the value of $\mbf{u}^b_j$ in terms of $\mbf{R}^b_j$ and the terms appearing at the boundary. To do this, we need to prove bounds on the individual terms appearing in $\mbf{R}^b_j$.  The only difference compared to \cite{paddick} and \cite{grenier} is in the terms 
\begin{equation}\label{eq:quantity}
    \frac{u_s^{\nu}(\sqrt{\nu}t,y)-U_s(y)}{\nu^{2^{-n}}}, \qquad \frac{u_s^{\nu}(\sqrt{\nu}t,\nu^{1/4}Y)}{\nu^{2^{-n}}},
\end{equation}
which need to be bounded uniformly in $\nu$. More precisely, he first term needs to be bounded in $H^s$ for all $s\geq 0$. For the second term a pointwise bound of the derivatives for all $t,Y\geq 0$ with slow growth as $Y\to +\infty$ suffices. Indeed we are ultimately multiplying these terms by a rapidly decaying function in $Y$, therefore some growth in $Y$ is allowed.
\begin{lemma}\label{usbound} Assume that $U_s \in C_b^{\infty}(\mathbb{R}_+)$, and $U_s$ satisfies \eqref{eq:assumpt}. Then the function
\begin{equation} 
    \nu \mapsto \frac{u_s^{\nu}(\sqrt{\nu}t,y)-U_s(y)}{\nu^{2^{-n}}}
\end{equation}
is bounded in $H^s(\mathbb{R}_+)$ for all $s\geq 0$, uniformly as $\nu \to 0$.
\end{lemma}

\begin{proof} Replacing $\nu=0$ in $u_s^{\nu}$, we know that
$$\frac{u_s^0(\sqrt{\nu}t,y)-U_s(y)}{\sqrt{\nu}}=\partial_{yy} u_s^0(\sqrt{\nu}\tau,y), \quad \tau=\tau(t) \in [0,1],$$
where $u_s^0$ satisfies the Dirichlet problem for the heat equation. But
$$\|\partial_{yy}u_s^0(\sqrt{\nu}\tau)\|_{H^s}=\|K(\sqrt{\nu}\tau)\star \tilde{U}_s''\|_{H^s}\leq \|\tilde{U}_s''\|_{H^s}, $$
where $\tilde{U}_s$ is the odd (if $\gamma >1/2$) or even (if $\gamma <1/2$) extension of $U_s $ from $\mathbb{R}_+$ to $\mathbb{R}$,
which is independent from $\nu$. Notice that $n\geq 2$, so $2^{-n}\leq 1/2$.

There are two cases. If $\gamma >1/2$ we want to show that
$\dfrac{(u_s^{\nu}-u_s^0)(\sqrt{\nu}t,y)}{\nu^{\gamma-1/2}}$
is uniformly bounded in $H^s$, which allows us to conclude since $2^{-n}\leq \gamma-1/2$ by definition.
But we know from Corollary \ref{conv} that 
$$\|u_s^{\nu}(t)-u_s^0(t)\|_{H^s}\leq C_{s}\nu^{\gamma-1/2},\qquad \forall t\geq 0, $$
 which immediately provides the desired result.
 
 If $\gamma <1/2$ we replicate the above argument except we consider $\dfrac{u_s^{\nu}-u_s^0}{ \nu^{1/4-\gamma/2}}$, and apply Corollary \ref{convneu}.
\end{proof}
\begin{remark} For profiles that do not satisfy assumption \eqref{eq:assumpt}, the above estimate does not hold. For instance, take $\gamma>1/2$, $U_s(y)=1$ for all $y\geq 0$, so that $U_s''(y)=0$. Then letting $\alpha:=\nu^{1/2-\gamma}$, we have for all $\alpha >0$,
$$\tilde{U}_s^{\alpha}(y)= \chi_{[0,\infty)}+\l -1 + 2\alpha e^{\alpha y}\r\chi_{(-\infty,0]}. $$
The evolution of the profile is given by
$$u_s^{\alpha}(t,y)=\textup{Erf}\l \frac{y}{2\sqrt{t}}\r + e^{\alpha(\alpha t+y)}\textup{Erfc}\l \frac{2\alpha t+y}{2\sqrt{t}}\r , $$
and its second derivative is
$$(u_s^{\alpha})''(t,y)= \alpha\l -2K(t,y)+ \alpha e^{\alpha(\alpha t+y)}\textup{Erfc}\l \frac{2\alpha t+y}{2\sqrt{t}}\r\r.  $$
The second term converges exponentially quickly to $0$ in all $L^p$ norms as $\alpha \to 0$. However, the term $-2\alpha K(t,y)$ does not - for instance, $\|K(\sqrt{\nu}t,y)\|_{L^2} \sim \nu^{-1/8}$. This term is the result of the Dirac delta appearing in $(\tilde{U}_s^{\alpha})''$. If $U_s$ satisfies \eqref{eq:assumpt}, then $\tilde{U}_s^{\alpha}$ is smooth and such singularities cannot occur.
\end{remark}
\begin{lemma}\label{otherbound} Assume that $U_s \in C_b^{\infty}(\mathbb{R}_+)$, and $U_s$ satisfies \eqref{eq:assumpt}. Then \begin{gather}\label{eq:otherbound}
    \left|\frac{u_s^{\nu}(\sqrt{\nu}t,\nu^{1/4}Y)}{\nu^{2^{-n}}}\right|\leq C_1Y+C_2\qquad \forall Y\geq 0,\\
    \label{eq:otherbound2}
    \left|\frac{\partial_Y^ku_s^{\nu}(\sqrt{\nu}t,\nu^{1/4}Y)}{\nu^{1/4}}\right|\leq C_3,\qquad \forall Y\geq 0,\,k\geq 1,
\end{gather}
where $C_1,C_2,C_3>0$ do not depend on $t\geq 0$ or on $\nu \to 0$.
\end{lemma}
 \begin{proof}We prove the result for $\gamma >1/2$. 
\begin{align*}
     \left|{u_s^{\nu}(\sqrt{\nu}t,\nu^{1/4}Y)}\right|&\leq |u_s^{\nu}(\sqrt{\nu}t,0)|+ \nu^{1/4}Y \|\partial_y u_s^{\nu}(\sqrt{\nu}t)\|_{L^{\infty}}\\ 
     &\leq |u_s^0(\sqrt{\nu}t,0)|+ \nu^{1/4}Y \|\partial_y u_s^{0}(\sqrt{\nu}t)\|_{L^{\infty}}+ O(\nu^{\gamma-1/2})\\
     &\leq \nu^{1/4}Y \|\partial_y U_s^0\|_{L^{\infty}}+ O(\nu^{\gamma-1/2}).
\end{align*}
Hence, \eqref{eq:otherbound} follows. For \eqref{eq:otherbound2}, recall that $Y=\nu^{-1/4}y$ and therefore
\begin{align*}
    \partial_Y^ku_s^{\nu}(\sqrt{\nu}t,\nu^{1/4}Y)&=\nu^{k/4}\partial_y u_s^{\nu}(\sqrt{\nu}t,y)\\ 
    &=\nu^{k/4}\l \partial_y u_s^0(\sqrt{\nu}t,y)+ O(\nu^{\gamma-1/2})\r\\ 
    &=\nu^{k/4}\|\partial_y U_s^0\|_{L^{\infty}}.
\end{align*}
For $\gamma <1/2$ the proof is the same where we replace $\gamma-1/2$ with $1/4-\gamma/2$, using Corollary \ref{convneu}. 
\end{proof}
\subsection{The inviscid linear instability}\label{inviscon}
We first start by constructing a term $\mbf{u}_0^I(t,x,y)$ displaying the instability. The construction follows from our assumption that the shear profile $U_s$ is linearly unstable for the linearized Euler equation, and proceeds exactly as in  \cite{grenier} or \cite{paddick}.  We sketch it here for the reader's convenience. 

Let us first specify exactly what we mean by linearly unstable. Consider the linearized Euler equation around the shear profile $U_s$:
\begin{equation*}\begin{cases}\partial_t \mbf{u}+\mbf{u}\cdot \nabla U_s+U_s\cdot \nabla \mbf{u}+\nabla p = 0;\\
v=0 & y=0.\end{cases}\end{equation*}
Then there exists a (nontrivial) solution $\mbf{u}$ in the form
$$\mbf{u}(t,x,y)=e^{ik(x-ct)}(\phi'(y),-ik\phi(y)),\qquad k\in\mathbb{R},c\in \mathbb{C}, $$
if and only if $\phi$ is a (nontrivial) solution of the Rayleigh equation
\begin{equation}\label{eq:ray} \begin{cases}(U_s-c)(\partial_{yy}-k^2)\phi - U_s'' \phi=0;\\ 
\phi(0)=\lim_{y\to +\infty}\phi(y)=0.
\end{cases}\end{equation}
\begin{deff}\label{defunst}
For each fixed $k \in \mathbb{R}$ we say that $c$ is an \emph{eigenvalue} for the Rayleigh equation if there exists a nontrivial solution to the Rayleigh equation \eqref{eq:ray}. The shear profile $U_s$ is called \emph{linearly unstable for the Euler equation} when the associated Rayleigh equation admits an eigenvalue $c$ with $\Im c>0$. 
\end{deff}

For each wavenumber $k\in \mathbb{R}$, let $\sigma(k)$  be the supremum of the real parts of the associated eigenvalues of the Rayleigh equation. By Theorem 4.1 from \cite{grenier}, this supremum is always attained at some eigenvalue $\lambda_k\in \mathbb{C}$; moreover,  $k\mapsto \sigma(k)$ is real analytic, non-negative, and
$$\lim_{k \to 0}\sigma(k)=\lim_{|k|\to +\infty}\sigma(k)=0. $$
In particular, $\sigma$ admits a maximum $\sigma_0=\sigma(k_0)\geq 0$ over $\mathbb{R}$. 
Since $U_s$ is by assumption linearly unstable for Euler, $\sigma(k)$ is not identically zero, so $\sigma_0>0$.  Because $k\mapsto \sigma(k)$ is continuous, we have $\sigma(k)>0$ in a neighborhood $I$ of $k_0$. Thus for all $k\in I$, we have a maximally unstable solution of the Euler equation
$$\mbf{u}_{k}(t,x,y)=e^{ikx + \lambda_k t}(\psi_k'(y),-ik\psi_k(y)), $$
where $\psi_k$ solves the Rayleigh equation with wavenumber $k$ and eigenvalue $\lambda_k$. We thus define 
$$\mbf{u}_0^I(t,x,y):=\int_{\mathbb{R}}\varphi(k)\mbf{u}_{k}(t,x,y)\,\mathrm{d}k,$$
where $\varphi$ is supported in a small enough neighborhood $I'\subset I$ of $k_0$. Notice that thanks to this cut off, we have $\mbf{u}_0^I\in H^s(\mathbb{R}\times \mathbb{R}_+)$ for all $s\geq 0$. We remark that if the domain of the $x$ variables is bounded instead, e.g. $x\in \mathbb{T}$, we could simply define $\mbf{u}_0^I:=\mbf{u}_{k_0}$. 

We thus obtain
$$ \|\mbf{u}_0^I(t)\|_{H^s}^2\sim  \int_{I'}e^{2\sigma(k)t}\,\mathrm{d}k.$$
To estimate this integral, we can use a Taylor expansion of $\sigma(k)$ around $\sigma(k_0)$. Since $k_0$ is a maximum and $\sigma$ is real analytic and non-constant, we have
$$\sigma(k) \sim \sigma_0-\mu \sigma^{(2m)}(k_0)(k-k_0)^{2m}, $$
for some $\mu>0$ and $m\geq 1$. Therefore,
$$\int_{I'}e^{2\sigma(k)t}\,\mathrm{d}k\sim e^{2\sigma_0 t}\int_{I'}e^{-2\mu t(k-k_0)^{2m}}\,\mathrm{d}k \sim C\frac{e^{2\sigma_0 t}}{t^{1/2m}},\qquad \text{ as } t\to +\infty. $$
In the remainder of the argument, in line with \cite{paddick} we just assume $m=1$ (i.e. $\sigma_0$ is a nondegenerate maximum), in order to simplify the notation. All the results still hold for arbitrary $m$.

In this case, for all $s\geq 0$ there exists $C_s>0$ such that
\begin{equation}\label{eq:u0bound}\|\mbf{u}_0^I(t)\|_{H^s}\leq C_s\frac{e^{\sigma_0 t}}{(1+t)^{1/4}},\qquad \forall t\geq 0.\end{equation}
    Additionally (see \cite{paddick}, Section 3.2.2), there exists a bounded subset $\Omega_A\subset \mathbb{R}^2_+$, with measure of order $\sqrt{1+t}$ such that
\begin{equation}\label{eq:lowerbound}\|\mbf{u}_0^I\|_{L^2(\Omega_A)}\geq C' \frac{e^{\sigma_0 t}}{(1+t)^{1/4}},\qquad \forall t\geq 0.\end{equation}

Assuming $N\geq 1$ and $\nu\leq 1$, we can define times $T^{\nu}_{\theta}>0$ such that 
\begin{equation}\label{eq:time} \frac{e^{\sigma_0 T^{\nu}_{\theta}}}{(1+T^{\nu}_{\theta})^{1/4}}=\nu^{\theta-N}.\end{equation}
Notice that $\lim_{\nu\to 0^+}T^{\nu}_{\theta}=\infty$, but in the original variables, $\lim_{\nu\to 0^+}\sqrt{\nu}T^{\nu}_{\theta}= 0$. Moreover, for $\tau>0$ small enough depending on $\nu$, we have $T^{\nu}:=T_{\theta}^{\nu}-\tau>0$, and by \eqref{eq:lowerbound} 
$$\|\left.\nu^N \mbf{u}^I_0\right|_{t=T^{\nu}}\|_{L^2(\Omega_A)}\geq \delta(\tau) \nu ^{\theta}. $$
This proves \eqref{eq:idk14}, for some value $\delta$ depending on the choice of $\tau$. In other words, we can always subtract a value $\tau$ as large as we want from $T^{\nu}_{\theta}$, and the instability will hold at $t=T^{\nu}_{\theta}-\tau$, as long as $\nu$ is small enough. We will choose the specific value of $\tau>0$ later on.
\subsection{Correction of the boundary condition}\label{bdrycon}
As discussed above, we want $\mbf{u}^{\textup{app}}$ to satisfy the Navier boundary condition \eqref{eq:nsnavier2}$_3$. But since it is constructed using an asymptotic expansion, we cannot in principle expect it to be satisfied exactly. Instead, it will leave a remainder $\mbf{r}^\textup{app}=(r_1^\textup{app},r_2^\textup{app})$, which will be determined in this subsection.

Let us first consider the equation for the first component:
$$\partial_y (u_s^{\nu}+u^I+\nu^a u^b)-\nu^{1/2-\gamma}(u_s^{\nu}+u^I+{\nu}^a u^b)=0\qquad \text{ at }y=0. $$
 By assumption, the shear flow $u_s^{\nu}$ satisfies the Navier condition $\partial_y u_s^{\nu}-{\nu}^{1/2-\gamma}u_s^{\nu}=0$, so it can be eliminated from the above condition and we are left with
\begin{equation*}u^b = {\nu}^{\gamma-3/4}\partial_Y u^b +{\nu}^{\gamma-1/2-a}\partial_y u^I-{\nu}^{-a}u^I.\end{equation*}\begin{itemize}
    \item If $\gamma >3/4\implies a=0$, the above relation reduces to
\begin{equation}\label{eq:navierub}u^b  = {\nu}^{\gamma-3/4}\partial_Y u^b + {\nu}^{\gamma-1/2}\partial_y u^I-u^I,\end{equation}
and as $\nu \to 0$ to just
$$u^b=-u^I,$$
so $u^b$ satisfies a Dirichlet problem

 \item If $\gamma=3/4\implies a=0$,  we get
$$\partial_Y u^b-u^b= u^I -\nu^{1/4}\partial_y u^I,$$
so in the limit $\nu\to 0$, $u^b$ satisfies a Robin problem.
\item If $\frac{1}{2}<\gamma < 3/4 \implies a=3/4-\gamma$, we obtain 
\begin{equation}\label{eq:navierub2}\partial_Y u^b=\nu^{3/4-\gamma}u^b +u^I -{\nu}^{\gamma-1/2}\partial_y u^I, \end{equation}
and as $\nu\to 0$ to just
$$\partial_Y u^b = u^I,$$
so $u^b$ satisfies a Neumann problem.

\item If $\gamma <1/2\implies a=1/4$,  we obtain 

$$\partial_Y u^b= \nu^{3/4-\gamma}u^b+\nu^{1/2-\gamma}u^I-\partial_y u^I,$$
as $\nu \to 0$ we get
$$\partial_Y u^b=-\partial_y u^I, $$
which is again a Neumann problem. 
\end{itemize}
Notice that regardless of the value of $\gamma$, the terms with $u^I$ do not appear at a higher order compared to the terms with $u^b$. This means that we can obtain the same estimates for $u^b_j$ as we do for $u^I_j$. It would no longer be true if the value of $a$ was higher.

Let us plug the ansatze \eqref{eq:ansatz2} for $u^I$ and $u^b$ into the boundary conditions. Then we can derive recursive equations which determine the value of $u^b_j$. As discussed in Section \ref{heat}, if  $u^b_j$ or $\partial_y u^b_j$ appear with a coefficient vanishing with the viscosity, the solution cannot be bounded uniformly with respect to the viscosity. Hence we need to move such term in the next order. This is possible in each case thanks to our choice of $n$. In the following, each term with a negative index should be replaced with $0$. Note that $3/4-\gamma-2^{-n}\geq 0$ for all $\gamma > 3/4$, by definition of $n$. 
\begin{itemize}
    \item If $\gamma>3/4$: each  $u^b_j$ solves the Dirichlet problem
    \begin{equation}\label{eq:bdry3}
           u^b_j = \nu^{\gamma-3/4-2^{-n}}\partial_Yu^b_{j-1}+u^I_j-\nu^{\gamma-1/2}\partial_y u^I_{j}.
    \end{equation}
    The final error in the boundary condition will be
    \begin{equation}\label{eq:err0}
        r_1^{\textup{app}}=\nu^{\gamma-3/4-2^{-n}+M2^{-n}}\partial_Y {u}^b_M.
    \end{equation}
    
    \item If $\gamma=3/4$: each  $u^b_j$ solves the Robin problem
    \begin{equation}\label{eq:bdry2}
        \partial_Y u^b_j -u^b_j = u^I_j - \nu^{1/4}\partial_y u^I_j.
    \end{equation}
    Since all the terms appear with the same order, there will be no error in the boundary condition.
    \item If $1/2<\gamma<3/4$: each  $u^b_j$ solves the Neumann problem
    \begin{equation}\label{eq:bdry1}
        \partial_Y u^b_j= \nu^{3/4-\gamma-2^{-n}}u^b_{j-1}+u^I_j -\nu^{\gamma-1/2}\partial_y u^I_j.
    \end{equation}
    The final error in the boundary condition will be 
    \begin{equation}\label{eq:err3}r_1^{\textup{app}}= \nu^{3/4-\gamma-2^{-n}+M2^{-n}}{u}^b_M.\end{equation}
    Notice that in the limit case $\gamma=1/2$, we do recover the boundary problem considered in \cite{paddick}.
    \item If $\gamma <1/2$, each $u^b_j$ solves the Neumann problem
\begin{equation}\label{eq:bdry4}\partial_Y u^b_j = \nu^{3/4-\gamma-2^{-n}}u^b_{j-1}+\nu^{1/2-\gamma}u^I-\partial_y u^I,\end{equation}
and the error in the boundary condition is again given by \eqref{eq:err3}. 
\end{itemize}
Notice that, for each value of $\gamma$, the terms $u_{j'}^I$ only appear in the boundary condition for $u_j^b$ with indices $j'\leq j$. This allows us to construct $\mbf{u}_j^b$ starting from $\mbf{u}_{j'}^I$, $j'\leq j$.

As for the second component $v_j^b$, from \eqref{eq:internal}$_3$  it follows that for all $j\leq M$ we have
$$ v^I_j+\nu^{a+1/4-2^{-n}} v^b_{j-1}=0\qquad \text{ at }y=0. $$
Therefore,
\begin{equation}\label{eq:remsecond}v^{\textup{app}}=\nu^{N+a+1/4-2^{-n}+M2^{-n}}v_{M}^j=:r_2^{\textup{app}}\qquad \text{ at }y=0. \end{equation}
\subsection{Construction of $\mbf{u}^I_j$ and $\mbf{u}^b_j$}

We will  now $\mbf{u}_j^I$ and $\mbf{u}_j^b$ by induction on $j\in \mathbb{Z}_{\geq 0}$. The induction is organized as follows. We start from $\mbf{u}_0^I$, which was constructed in Section \ref{inviscon}. From $\mbf{u}_0^I$ we can construct $\mbf{u}_0^b$, applying the boundary condition derived in Section \ref{bdrycon}. Next, suppose we have constructed $\mbf{u}_{j'}^I$ and $\mbf{u}_{j'}^b$ for all $0\leq j'\leq j$. In the equation \eqref{eq:internal} satisfied by $\mbf{u}_{j+1}^I$, $\mbf{R}^I_{j+1}$ only depends on $\mbf{u}_j'^I$ for $j'\leq j$, while the boundary condition depends on $\mbf{u}_{j}^b$, all of which have already been constructed. Thus we can derive $\mbf{u}_{j+1}^I$. Similarly, by \eqref{eq:bdryj}, all the terms in $\mbf{R}_{j+1}^b$ only depend on $\mbf{u}^b_{j'}$ and $\mbf{u}^I_{j'}$ for $j'<j+1$, and the same goes for the boundary condition (see Section \ref{bdrycon}). Thus we obtain $\mbf{u}^b_{j+1}$. By induction, we can construct $\mbf{u}_j^I$ and $\mbf{u}_j^b$ for all $j=0,\dots,M$.

In what follows, to ease the notation we introduce for all $j=1,\dots,M$ the quantity
$$k_j:= 1+ \frac{j}{2^nN}. $$
\begin{prop} For all $s\geq 0$ there exists a constant $C=C(s,j)>0$ such that for all $t\geq 0$, and $j=0,\dots, M$, we have
\begin{align}\label{eq:intj}\|\mbf{u}_j^I(t)\|_{H^s}&\leq C \frac{e^{\sigma_0 k_j t}}{(1+t)^{k_j/4}}.\\
\label{eq:jbest} \|\mbf{u}^b_j(t)\|_{H^s}&\leq C \frac{e^{\sigma_0 k_j t}}{(1+t)^{k_j/4}}, \qquad \forall t\geq 0. \end{align}
Moreover, there exist functions $h_{k,\ell,j}(x)\in L^2(\mathbb{R}),\mu_j>0$ such that for all $k,\ell\in \mathbb{Z}_{\geq 0}$, and for all $t\geq 0,(x,y)\in \mathbb{R}\times \mathbb{R}_+$,
\begin{equation}\label{eq:jest2} |\partial_x^k\partial_Y^{\ell}\mbf{u}^b_j(t,x,Y)|\leq |h_{k,\ell,j}(x)| \frac{e^{\sigma_0 k_j t}}{(1+t)^{k_j/4}} e^{-\mu_j Y}, \qquad j=0,\dots,M. \end{equation}
\end{prop}
\begin{proof}
The proof is by induction. The estimate \eqref{eq:intj} for $j=0$ is simply \eqref{eq:u0bound}, whereas \eqref{eq:jest2} and \eqref{eq:jbest} follow from Proposition \ref{pointwise}. Now suppose \eqref{eq:intj} and \eqref{eq:jbest} hold for $j<J$. We first look to obtain \eqref{eq:intj} for $j=J$. We know that $\mbf{u}_J^I$ satisfies \eqref{eq:internal} and \eqref{eq:intremj}. We want to find an $H^s$ estimate on the remainder $\mbf{R}_J^I$. Recall that, for all $f,g\in H^s$, for any $s\geq 0$ by Sobolev embeddings we have
$$\|fg\|_{H^s}\leq C_s \l \|f\|_{H^s}\|g\|_{L^{\infty}}+\|f\|_{L^{\infty}}\|g\|_{H^s}\r\leq C \|f\|_{H^{s+2}}\|g\|_{H^{s+2}}. $$
Since

Therefore, using Corollary \ref{usbound}, 
\begin{align*}
    \|\mbf{R}_J^I\|_{H^s}&\lesssim \|\mbf{u}_{J-1}^I\|_{H^{s+1}}+ \|\mbf{u}_{J-2}^I\|_{H^{s+2}}+\sum_{j_1+j_2=J-2^nN} \|\mbf{u}_{j_1}^I\|_{H^{s+2}}\|\mbf{u}_{j_2}^I\|_{H^{s+2}}\\
    &\lesssim \|\mbf{u}_{j-1}^I\|_{H^s}+\frac{e^{\sigma_0(k_{j_1}+k_{j_2})t}}{(1+t)^{(k_{j_1}+k_{j_2})/4}}\\
    &\lesssim  \frac{e^{\sigma_0 k_J t}}{(1+t)^{k_J/4}},
\end{align*}
where we used the equality
$$k_{j_1}+k_{j_2}=k_{j_1+j_2+2^nN}. $$

Finally, using well-known spectral estimates on the linearized Euler equations (see Theorem 3.1 from \cite{paddick}), since $\mbf{u}_{J-1}^b$ satisfies \eqref{eq:jbest}, we deduce that 
$$\|\mbf{u}_{J}^I\|_{H^{s-2}}\lesssim \|\mbf{R}_J^I\|_{H^s}\lesssim \frac{e^{\sigma_0 k_J t}}{(1+t)^{k_J/4}}. $$
Since the estimate works for all $s\geq 0$, we conclude that \eqref{eq:intj} holds for the index $j$.

Next, we look to prove the estimates for $\mbf{u}^b_J$. It is enough to prove \eqref{eq:jest2} as \eqref{eq:jbest}  immediately follows from it. Moreover, it is enough to prove \eqref{eq:jest2} for the first component $u^b_J$, as we can then deduce them for the second component $v^b_J$ by the divergence-free condition, as in \eqref{eq:divfree}. Recall that ${u}_J^b$ satisfies \eqref{eq:bdry} and \eqref{eq:bdryj} with boundary conditions \eqref{eq:bdry3}, \eqref{eq:bdry1} or \eqref{eq:bdry2}.  For each $x\in \mathbb{R}$, this is a heat equation in $Y$ with a remainder  ${R}_J^b$ as in \eqref{eq:bdryj}, an initial condition satisfying \eqref{eq:icbound},
and a boundary condition given by a sum of $u^I_{j},u^b_j$ and their derivatives for $j\leq J$. For the $u^I_j$ terms, by taking the trace at the boundary in \eqref{eq:intj} we deduce that
$$|\partial_x^{\ell}\partial_Y^k {u}_j^I(t,x,0)|\leq  |f_{k,\ell,J}(x)|\frac{e^{\sigma_0 k_J t}}{(1+t)^{k_J/4}},$$
where $f_{k,\ell,J}(x)\in L^2(\mathbb{R})$. Together with the inductive assumption \eqref{eq:jest2}, we deduce that the inhomogeneous part of the boundary condition as a whole satisfies the same estimate.

Now consider the remainders $R_J^b$. The only potential danger is the terms in the remainders containing $u_s(\sqrt{\nu}t,y)\nu^{-1/8}$ which need to be replaced with $u_s^{\nu}(\sqrt{\nu}t,y)\nu^{-2^{-n}}$. More precisely these terms are of the form 
\begin{align}\label{eq:start}&\frac{u_s^{\nu}(\sqrt{\nu}t,\nu^{1/4}Y)}{\nu^{2^{-n}}}\cdot \nabla_{x,Y}{u}_{J-1}^b+{u}_{J-1}^b\cdot \nabla_{x,Y} \l \frac{u_s^{\nu}(\sqrt{\nu}t,\nu^{1/4}Y)}{\nu^{2^{-n}}}\r 
\end{align}

By Proposition \ref{otherbound}, since $n\geq 2$, these terms and their derivatives can be bound pointwise by
\begin{equation}\label{eq:terms}|C_1Y+C_2|\cdot|\nabla_{x,Y}{u}_{J-1}^b(t,x,Y)|+ C_3\cdot  |{u}_{J-1}^b(t,x,Y)|,\end{equation}
Thus assuming ${u}_{J-1}^b$ satisfies \eqref{eq:jest2}, then the terms in \eqref{eq:start} also do, after decreasing the value of $\mu_{J}$ by an arbitrarily small quantity to accomodate for the extra linear growth in $Y$. Therefore the remainder $R_J^b$ satisfies the estimate \eqref{eq:jest2}. Since the estimate holds for $Y$-derivatives of all orders, we deduce that
$$|\partial_x^{\ell}\partial_y^k{R}^b_J(t,x,Y)|\leq |R_{k,\ell,J}(x)|\frac{e^{\sigma_0 k_j t}}{(1+t)^{k_j/4}}e^{-\mu_{J}'Y}, $$
where $R_{k,\ell,J}\in L^2(\mathbb{R})$. By Proposition \ref{pointwise}, we conclude that
$$|\partial_x^{\ell}\partial_y^k{u}^b_{J}(t,x,Y)|\leq |h_{k,\ell,J}(x)|\frac{e^{\sigma_0 k_{J} t}}{(1+t)^{k_{J}/4}}e^{-\mu_{J}Y}, $$
where, given the bound $g_j^b(x)\in L^2(\mathbb{R})$ for the initial conditions of $\mbf{u}_j^b$ as in \eqref{eq:icbound},
$$|h_{k,\ell,J}(x)|\leq  C\l |R_{k,\ell, J}(x)|+ |f_{k,\ell,J}(x)|+|g_j^b(x)|\r,$$
and thus $h_{k,\ell,J}(x)\in L^2(\mathbb{R})$.

Hence, by induction, \eqref{eq:jest2} is verified for all $0\leq J\leq M$. Since it holds for derivatives of all orders, the $H^s$ estimate immediately follows. 
\end{proof}
\begin{corollary}For all $s\geq 0$, we have
\begin{equation}\label{eq:intrem}
\|\mbf{R}^I\|_{H^s}\leq C_s \l  \nu^{N}\frac{e^{\sigma_0 t}}{(1+t)^{1/4}}\r ^{1+\frac{M+1}{2^nN}}\qquad \forall t\leq T^{\nu}_{\theta}.
\end{equation}
\end{corollary}
\begin{proof}By \eqref{eq:intj} and \eqref{eq:intfinal}, we know that there exists an $M'>M$ such that
\begin{align*}\|\mbf{R}^I(t)\|_{H^s}&\leq C_s\nu^N\sum_{j=M+1}^{M'}\nu^{j2^{-n}}\frac{e^{\sigma_0 k_j t}}{(1+t)^{k_j/4}}=C_s \sum_{j= M+1}^{M'}\l\nu^{N} \frac{e^{\sigma_0  t}}{(1+t)^{1/4}}\r^{k_j}\\ 
&\leq C_s\l\nu^N\frac{e^{\sigma_0  t}}{(1+t)^{1/4}}\r^{1+\frac{M+1}{2^nN}}+C_s\sum_{j= M+2}^{M'}\nu^{\theta k_j}\\ 
&\leq  C_s\l\nu^N\frac{e^{\sigma_0  t}}{(1+t)^{1/4}}\r^{1+\frac{M+1}{2^nN}}.\end{align*}
\end{proof}
This estimate will be necessary in order to prove corresponding estimates on the approximate solution $\mbf{u}^{\textup{app}}$.
 
In the same way as for $\mbf{R}^I$ in \eqref{eq:intrem}, we deduce the following estimate for $\mbf{R}^b$.
\begin{corollary} For all $\ell,k\in \mathbb{Z}_{\geq 0}$ there exist constants $C_{k,\ell},\mu>0$ such that for all $t\leq T^{\nu}_{\theta}$ we have
\begin{equation}\label{eq:rb}|\partial_x^\ell \partial_Y^k \mbf{R}^b(t,x,Y)|\leq C_{k,\ell} \l\nu^{N}\frac{e^{\sigma_0 t}}{(1+t)^{1/4}} \r^{1+\frac{M+1}{2^nN}}e^{-\mu y}.\end{equation}
\end{corollary}
Putting \eqref{eq:intrem} and \eqref{eq:rb} together, we obtain the estimates for the remainder $\mbf{R}^{\textup{app}}=\mbf{R}^I + \nu^a \mbf{R}^b$: for all $s\geq 0$, we have
\begin{equation}\label{eq:finalrem}
    \|\mbf{R}^{\textup{app}}(t)\|_{H^s}\leq C_s \l\nu^{N}\frac{e^{\sigma_0 t}}{(1+t)^{1/4}} \r^{1+\frac{M+1}{2^nN}}.
\end{equation}
Now let us consider the remainder $\mbf{r}^{\textup{app}}$ in the Navier boundary condition.  For the first component, we have by \eqref{eq:jbest}, 
\begin{equation}\label{eq:rapp}\|r_1^{\textup{app}}(t)\|_{L^2(y=0)}\leq C_{M} \nu^{N+\frac{M+1}{2^n}}\frac{e^{\sigma_0 k_{M+1}t}}{(1+t)^{k_{M+1}/4}}=C_M\l \nu^N \frac{e^{\sigma_0 t}}{(1+t)^{1/4}} \r^{1+\frac{M+1}{2^nN}}. \end{equation}
If $\gamma=3/4$ the above inequality is trivially true since $r_1^{\textup{app}}=0$. When $\gamma <\frac{1}{2}$, by \eqref{eq:err3} we further have 
\begin{equation}\label{eq:rapp2}\nu^{\gamma-3/4+2^{-n}}\|r_1^{\textup{app}}(t)\|_{L^2(y=0)}\leq C_M\l \nu^N \frac{e^{\sigma_0 t}}{(1+t)^{1/4}} \r^{1+\frac{M+1}{2^nN}}. \end{equation}
Notice that $\gamma-3/4+2^{-n}\leq \gamma-1/2<0$. 

For the second component $r_2^{\textup{app}}$, by \eqref{eq:remsecond}, discarding the $\nu^{a+1/4-2^{-n}}$ factor we similarly obtain, using \eqref{eq:jest2},
\begin{equation}\label{eq:r2app}|\partial_x^{\ell}\partial_y^k r_2^{\textup{app}}(t,x,0)| \leq C_{k,\ell} \l \nu^N \frac{e^{\sigma_0 t}}{(1+t)^{1/4}} \r^{1+\frac{M+1}{2^nN}}.\end{equation}

\subsection{Energy estimates}\label{energyestt}
Define $\mbf{v}:=\mbf{u}^{\nu}-\mbf{u}^{\textup{app}}$. Then $\mbf{v}+\mbf{u}^{\textup{app}}$ solves Navier-Stokes with an error $\mbf{R}^{\textup{app}}$, so that $\mbf{v}$ solves the equation
\begin{equation*}\begin{cases}\partial_t \mbf{v} + (\mbf{u}^{\textup{app}}\cdot \nabla) \mbf{v} + (\mbf{v}\cdot \nabla) \mbf{u}^{\textup{app}} + (\mbf{v}\cdot \nabla) \mbf{v}+ \nabla p = \sqrt{\nu}\Delta \mbf{v}-\mbf{R}^{\textup{app}};\\ 
\nabla \cdot \mbf{v} = 0; \\
\left.\mbf{v}\right|_{t=0}=0;
\end{cases} \end{equation*}
with some boundary condition which we will specify later. \\

We want to prove \eqref{eq:idk12}, so we need to find an upper bound on the $L^{2}$ norm of $v$ at time $t=T^{\nu}$. Deriving the standard energy estimate:
\begin{align*}
    \frac{1}{2}\partial_t \|\mbf{v}\|_{L^2}^2 + \sqrt{\nu}\|\nabla \mbf{v}\|_{L^2}^2-\sqrt{\nu}\int_{\partial \Omega}\partial_i v_j v_j n_i\leq \l \|\nabla \mbf{u}^{\textup{app}}\|_{L^{\infty}}+\beta\r\|\mbf{v}\|_{L^2}^2+\frac{1}{4\beta}\|\mbf{R}^{\textup{app}}\|_{L^2}^2.
\end{align*}
Because the domain is flat, we see that
$$\int_{\partial \Omega}\partial_i v_jv_jn_i = -\int_{y=0}v_1\partial_y v_1-\int_{y=0}v_2\partial_yv_2. $$
Suppose now that $\mbf{u}^{\textup{app}}$ satisfies the Navier boundary condition with an error $\mbf{r}^{\textup{app}}$, i.e.
$$\begin{cases}
\partial_y u^{\textup{app}}= \sqrt{\nu}^{1-2\gamma} u^{\textup{app}}+ r_1^{\textup{app}};\\
v^{\textup{app}}=r_2^{\textup{app}}.\end{cases} $$
Then $\mbf{v}$ satisfies the same boundary condition with error $-r^{\textup{app}}$:
$$\begin{cases} 
\partial_y v_1= \sqrt{\nu}^{1-2\gamma} v_1-r_1^{\textup{app}};\\
 v_2 =-r_2^{\textup{app}}.\end{cases}$$
Hence, for any $\alpha>0$,
    \begin{align*}\int_{y=0}v_1\partial_y v_1 &= \sqrt{\nu}^{1-2\gamma}\int_{y=0}|v_1|^2 - \int_{y=0}r^{\textup{app}}_1v_1\geq -\frac{1}{4}\nu^{\gamma-1/2}\int_{y=0}|r^{\textup{app}}_1|^2=(\star).\end{align*}
   
There are two cases. 
\begin{itemize}
    \item If $\gamma \geq 1/2$, then $\nu^{\gamma-1/2}\leq C$ for all $\nu$ small enough and the energy estimate becomes
\begin{align}\label{eq:energest}
    \frac{1}{2}\partial_t \|\mbf{v}\|_{L^2}^2+ \sqrt{\nu}\|\nabla \mbf{v}\|_{L^2}^2\leq  \l \|\nabla \mbf{u}^{\textup{app}}\|_{L^{\infty}}+\beta\r\|\mbf{v}\|_{L^2}^2+C\l \|\mbf{R}^{\textup{app}}\|_{L^2}^2+ \int_{y=0}|r_1^{\textup{app}}|^2\r-\int_{y=0}r_2^{\textup{app}}\partial_y r_2^{\textup{app}}.
\end{align}
Define
$$P:=1+\frac{M+1}{2^nN}.$$
Combining \eqref{eq:finalrem}, \eqref{eq:rapp} and \eqref{eq:r2app}, we have
\begin{equation}\label{eq:parts}\|\mbf{R}^{\textup{app}}(t)\|^2+ \|\mbf{r}^{\textup{app}}(t)\|_{L^2(y=0)}^2+\int_{y=0}|r_2^{\textup{app}}\partial_y r_2^{\textup{app}}| \leq C_M \l \nu^{N} \frac{e^{\sigma_0  t}}{(1+t)^{1/4}}\r^{2P}. \end{equation}
\item If $\gamma < 1/2$, we have 
    $$0 > \gamma-\frac{1}{2}\geq 2\l \gamma-\frac{1}{2} \r \geq 2\l \gamma-\frac{3}{4}+2^{-n}\r,$$
    hence for $\nu \leq 1$,
    $$(\star) \geq -\frac{1}{4}\nu^{2(\gamma-3/4+2^{-n})}\int_{y=0}|r^{\textup{app}}_1|^2. $$
    We can then proceed as in the previous case and use \eqref{eq:rapp2} to obtain \eqref{eq:parts}.

\end{itemize}

Next, choose $M$ large enough so that
$$ P\sigma_0-1 \geq \|\nabla \mbf{u}^{\textup{app}}\|_{L^{\infty}}+\beta\qquad \forall \nu>0.$$
For this to work, we need $\|\nabla \mbf{u}^{\textup{app}}\|_{L^{\infty}}$ to be bounded uniformly in $\nu$. The only potential issue is with $\mbf{u}^b$, as it depends on $y/\sqrt{\nu}$. Using \eqref{eq:jbest} we have, for $t\leq T^{\nu}$ as defined in \eqref{eq:time},
\begin{align*} \|\nabla \mbf{u}^{\textup{app}}(t)\|_{L^{\infty}}-\|\partial_y u_s^{\nu}(\sqrt{\nu}t)\|_{L^{\infty}}&\leq \nu^{N+a-1/4}\sum_{j=0}^M \nu^{j2^{-n}}\frac{e^{\sigma_0 k_j t}}{(1+t)^{k_j/4}} \\ 
& \leq C\nu^{a-1/4}\sum_{j=0}^M \nu^{\theta\l 1+ \frac{j}{2^nN}\r}\leq C \nu^{\theta+a-1/4}.\end{align*}
The power of $\nu$ appearing above is non-negative if and only if $\theta \geq 1/4-a$. Hence $\theta=1/4-a$ is the best value we can get in the instability. Of course, $\|\partial_y u_s^{\nu}(\sqrt{\nu}t)\|_{L^{\infty}}$ is bounded uniformly in $\nu$ by Corollary \ref{conv} or Corollary \ref{convneu}.

Now \eqref{eq:energest} becomes
\begin{align*}\label{eq:energest2}
   \partial_t\|\mbf{v}(t)\|_{L^2}^2\leq  ( 2P\sigma_0-1) \|\mbf{v}(t)\|_{L^2}^2+C_M\l \nu^{N} \frac{e^{\sigma_0 t}}{(1+t)^{1/4}}\r^{2P}.
\end{align*}
To conclude the proof, we apply Lemma \ref{gronwall} with $\varphi(t)=\|w(t)\|_{L^2}^2$, $\lambda = 2P\sigma_0-1$, $\mu = 2P\sigma_0$. We obtain
$$\|\mbf{v}(t)\|_{L^2}=\|\mbf{u}^{\nu}(t)-\mbf{u}^{\textup{app}}(t)\|_{L^2}\leq C'_M \l \nu^N\frac{ e^{\sigma_0 t}}{(1+t)^{1/4}}\r^P. $$
Choosing $\tau$ large enough in the definition of $T^{\nu}=T^{\nu}_{\theta}-\tau$, the above quantity is smaller than $\delta \nu^{\theta}$ for $t\leq T^{\nu}$, and we have verified \eqref{eq:main1}; \eqref{eq:main2} follows by the embedding $L^{\infty}\hookrightarrow \dot{H}^s$ for $s>1$. Therefore, \autoref{main} is proven.

\bibliographystyle{unsrtnat}
 
\label{sec:others}






\end{document}